\documentclass[twoside,leqno,11pt]{article}

\usepackage{a4wide}

\usepackage{amsmath}
\usepackage{rotate,float}
\usepackage{color}
\usepackage{amsmath,amssymb,bbm,enumerate}
\usepackage{pstricks}

\usepackage{graphicx} 

\usepackage{pst-3d,float}
\newtheorem{theorem}{Theorem}[section]

\newtheorem{definition}{Definition}[section]

\newtheorem{lemma}{Lemma}[section]

\newtheorem{proposition}[theorem]{Proposition}
\newtheorem{remark}{Remark}[section]

\newenvironment{proof}{\medskip\noindent{\bf Proof.}\;}{\null\hfill $\Box$\par\medskip }

\title{The Monogenic Synchrosqueezed Wavelet Transform: A tool for the Decomposition/Demodulation of AM-FM images}
\author{ \small \sc M. Clausel, T. Oberlin, V. Perrier\thanks{University of Grenoble and CNRS, Laboratoire Jean Kuntzmann UMR 5224, Saint Martin d'H\`eres, France.}}

\def\rmd{\mathrm{d}}
\def\rme{\mathrm{e}}
\def\rmi{\mathrm{i}}
\def\rmj{\mathrm{j}}
\def\rmk{\mathrm{k}}
\def\1{\mathbbm{1}}
\def\R{\mathbb{R}}

\begin{document}
\maketitle
\begin{abstract}
The synchrosqueezing method aims at decomposing 1D functions as superpositions of a small number of ``Intrinsic Modes'', supposed to be well separated both in time and frequency. Based on the unidimensional wavelet transform and its reconstruction properties, the synchrosqueezing transform provides a powerful representation of multicomponent signals in the time-frequency plane, together with a reconstruction of each mode.

In this paper, a bidimensional version of the synchrosqueezing transform is defined, by considering a well--adapted extension of the concept of analytic signal to images: the monogenic signal. The natural bidimensional counterpart of the notion of Intrinsic Mode is then the concept of ``Intrinsic Monogenic Mode'' that we define. Thereafter, we investigate the properties of its associated Monogenic Wavelet Decomposition. This leads to a natural bivariate extension of the Synchrosqueezed Wavelet Transform, for decomposing and processing multicomponent images. Numerical tests validate the effectiveness of the method for different examples.
\end{abstract}

\section{Introduction}
\label{sec:intro}
In the past few years, there has been an increasing interest in representing images using spatially varying sinusoidal waves. As the understanding of the theory advanced, amplitude- and frequency-modulation (AM--FM) decompositions have been applied in a large range of problems: for example motion estimation using a flow--optical method based on an assumption of phase--invariance \cite{fleet:jepson:1990,basarab:liebgott:delacharte:2009,basarab:gueth:liebgott:delacharte:2009}, reconstruction of breast cancer images \cite{elshinawy:chouikha:2003,elshinawy:zeng:lo:chouikha:2004}, texture analysis~\cite{kokkinos:evangelopoulos:maragos:2009,tay:2008}, or ultrasound image segmentation methods~\cite{belaid:boukerroui:maingourd:lerallut:2011} based on the
concept of monogenic signal and quadrature filters~\cite{felsberg:sommer:2001}. A survey of applications of AM--FM decompositions in medical imaging can also be found in~\cite{murray:2012}.

In each case, the main challenge is to decompose any input image $s(x)$ into a sum of bidimensional AM--FM harmonics of the form
\begin{equation}\label{e:decIMF2D}
s(x_1,x_2)=\sum_{\ell=1}^L s_{\ell}(t) = \sum_{\ell=1}^L A_\ell(x_1,x_2)\cos(\varphi_\ell(x_1,x_2))\;,
\end{equation}
where $A_\ell>0$ denotes a slowly--varying amplitude function, $\varphi_\ell$ denotes the
phase, and $\ell = 1,\cdots,L$ indexes the different AM--FM harmonics. To each phase function, one can associate an instantaneous frequency vector field defined as $\omega_\ell=\nabla\varphi_\ell$. Finding the components $s_\ell$ from the bidimensional signal $s$ is called the {\it decomposition problem}.

When the decomposition is given, another subsequent problem is to determine for each component $s_\ell$ its amplitude, phase and frequency functions involved in~(\ref{e:decIMF2D}). The amplitudes can be related to the energy contained at each point of the image, whereas significant texture variations are
captured in the frequency content. For example, in the case of a single component,
the instantaneous frequency vectors are orthogonal to the isovalue intensity
lines of an image, while their magnitudes provide a
measure of local frequency content. The problem of estimating the amplitudes, phases and instantaneous frequencies of a signal of the form (\ref{e:decIMF2D}) (bidimensional or not) is called the {\it demodulation problem}.

In the one dimensional case, these two problems have been widely investigated both from the theoretical and practical point of view. To address the demodulation problem in the one dimensional setting, a first and necessary step is to define, in a proper way, the concept of amplitude, phase and instantaneous frequency of a given signal. To this end, one can use the well--known Hilbert transform, defined for any $f\in L^2(\mathbb{R})$ as
\[
\mbox{For a.e.}~ t\in\mathbb{R},\,\mathcal{H}f(t)=\lim_{\varepsilon\to 0}\left(\frac{1}{\pi}\int_{|t-s|>\varepsilon}\frac{f(s)}{t-s}~\rmd s\right)\;.
\]
The analytic signal associated to $f$ is then the complex--valued function $F=(1+\rmi\mathcal{H})f$. Thereafter, one defines amplitude $A$ and phase $\varphi$ of $f$ as, respectively, the modulus and argument of the analytic signal $F$ associated to $f$ (the uniqueness of phase being ensured under smoothness assumptions on $f$ and under some initial conditions of the form $\varphi(0)=\varphi_0$). One then obtains
\[
\mbox{For a.e.} ~t\in\mathbb{R},\,F(t)=A(t)\rme^{\rmi \varphi(t)}\;.
\]
The instantaneous frequency of $f$ is thus the derivative of the phase $\varphi'$ \cite{qian:2006}. The practical estimation of amplitude, phase and instantaneous frequency was the subject of numerous studies (see~\cite{huang:2009} for a survey). The ``naive'' method, which consists in using the Hilbert transform to estimate $A$ and $\varphi$, is known to be numerically unstable \cite{rodriguez:pattichis:2005}. Alternative methods  use the wavelet transform associated to reallocation techniques: in the one dimensional case, the time--frequency localization of wavelets allows to compute robust estimations of instantaneous frequency lines in the time-frequency representation: several approaches have been developed in the 90th, one should cite the method of ``wavelet ridges'' (see the two pioneer works~\cite{carmona:hwang:torresani:1997} and~\cite{carmona:hwang:torresani:1999}), the squeezing method introduced in~\cite{daubechies:Maes:1996} or the reassignment  method~\cite{auger:flandrin:1995}.

On the other side, the decomposition problem consists in developing decompositions of any signal into a sum of ``well--behaved'' AM--FM components, and was widely studied. A recent attempt  in this context is Empirical Mode Decomposition (EMD), initially introduced by Huang et al. in~\cite{huang:1998} and later popularized by Flandrin and his co--authors (see e.g.~\cite{flandrin:rilling:goncalves:2004} or~\cite{rilling:flandrin:rilling:2008}).

%
%
%
%
Basically, EMD is the output of an iterative algorithm which provides an adaptive decomposition of any signal into several AM--FM components. It is now used in a wide range of applications including meteorology,
structural stability analysis, and medical studies \cite{huang:2009}. In spite of its simplicity and its efficiency, this method is hard to analyze mathematically, since it is defined in an empirical way. To tackle this problem, in~\cite{daubechies:2011}, the authors propose and analyze an alternative method called
``synchrosqueezing'' derived from reassignment methods: they introduce a class of functions which can be viewed as the superposition of a reasonable number of modes, i.e. which read $\sum_\ell A_\ell(t)\rme^{\rmi\varphi_\ell(t)}$, where the amplitude $A_\ell(t)$ of each mode is slowly varying with respect to its instantaneous frequency $\omega_\ell(t)=\varphi_\ell'(t)$. For each function belonging to this class, the synchrosqueezing provides an estimation of this (unknown) decomposition. Synchrosqueezing presents many similarities with Empirical Mode Decomposition as shown in~\cite{daubechies:2011} (see also~\cite{wu:flandrin:daubechies:2011} for a comparison of these two methods or~\cite{meignen2012new} for more about potential applications). In addition, since synchrosqueezing is based on the wavelet transform, it can be used to address simultaneously the decomposition and demodulation problem as done in~\cite{daubechies:2011}.

In this paper, we will focus on the two dimensional case and extend the synchrosqueezing approach to images. Our first step consists in defining a convenient extension of the concept of analytic signal, in order to properly define the notion of amplitude and phase of an image. Two main generalizations of the analytic signal to the bidimensional setting have been considered in the literature: the hypercomplex and the monogenic signal defined respectively in~\cite{bulow:sommer:1999} and~\cite{felsberg:sommer:2001}. Here, we focus on the monogenic setting, since its characteristics lead to easier interpretation. The monogenic signal associated to an image is defined using the Riesz transform. If we are given a real valued function $f\in L^2(\mathbb{R}^2)$, one defines its Riesz transform $\mathcal{R}f$ as the following vector--valued function
\[
\mathcal{R}f=\begin{pmatrix}\mathcal{R}_1f\\\mathcal{R}_2f\end{pmatrix}\;,
\]
where for any $i=1,2$ and a.e. $x\in\mathbb{R}^2$ one has
\[
\mathcal{R}_if(x)=\lim_{\varepsilon\to 0}\left(\frac{1}{\pi}\int_{|x-y|>\varepsilon}\frac{(x_i-y_i)}{|x-y|^3}f(y)~\rmd y\right)\;.
\]
Then, the monogenic signal associated to $f$ is
\[
\mathcal{M}f=\begin{pmatrix}f\\\mathcal{R}f\end{pmatrix}\;.
\]
The characteristics of the monogenic signal are usually defined using the quaternionic formalism (see Appendix~\ref{s:appendixQuater} for some background about quaternionic calculus and Appendix~\ref{s:appendixRiesz} for the main properties of the Riesz transform): the monogenic signal $\mathcal{M}f$ associated to $f$ can be read
\[
\mathcal{M}f=f+\mathcal{R}_1f~\rmi+\mathcal{R}_2f ~\rmj\;,
\]
where $(1,\rmi,\rmj,\rmk)$ denotes the canonical basis of the algebra $\mathbb{H}$ of the quaternions. In~\cite{felsberg:sommer:2001} and~\cite{yang:qian:sommen:2011}, it is proved that the monogenic signal can be written as
\[
\mathcal{M}f=A~\rme^{\varphi n_{\theta}}=A~(\cos\varphi+n_\theta\sin\varphi)\mbox{ with }n_\theta=\cos\theta~\rmi+\sin\theta~\rmj\;,
\]
where $A$ is a positive real valued function and $\varphi$ and $\theta$ are two real valued functions defined in a unique way (under suitable assumptions). The function $A$ is called the amplitude of $f$, $\varphi$ its phase and $n_\theta$ its local orientation. The function $\omega=\nabla\varphi$ will be called the instantaneous frequency of $f$. To illustrate these definitions, we consider the case where $f(x)=A_0\cos(k\cdot x)$ with $k=(k_1,k_2)$ and $k_1>0$. Set $\theta_0=\mathrm{Arctan}(k_2/k_1)$. One then has
\[
\mathcal{R}f(x_1,x_2)=A_0\begin{pmatrix}\sin(k_1 x_1+k_2 x_2)\cos\theta_0\\ \sin(k_1 x_1+k_2 x_2)\sin\theta_0\end{pmatrix}
= A_0 \frac{k}{|k|}\sin(k \cdot x)\;,
\]
and
\[
\mathcal{M}f(x)=\begin{pmatrix}f(x)\\\mathcal{R}f(x)\end{pmatrix}=A_0\rme^{(k\cdot x) (\cos\theta_0~\rmi+ \sin\theta_0~\rmj)}\;.
\]
Then for all $x\in\mathbb{R}^2$, one can set
\[
A(x)=A_0,\;\varphi (x)=k\cdot x,\;\theta(x)=\theta_0=\mathrm{Arctan}(k_2/k_1)\;.
\]

The aim of the present paper is to define the synchrosqueezing of bidimensional images, in the monogenic setting, in order to address both the decomposition and the demodulation problems. To this end, we will consider monogenic signals which are superposition of monogenic waves of the form
\begin{equation}\label{e:monogwaves}
A(x)\rme^{\varphi(x) n_{\theta(x)}}\;,
\end{equation}
where $A$ and $\theta$ are slowly varying functions with respect to $\varphi$.

We will consider successively the two different problems of deconvolution and decomposition. As in \cite{daubechies:2011}, our method is based on the continuous wavelet transform. Therefore, in Section~\ref{s:monogwav}, we recall a recent extension of the continuous bidimensional wavelet  transform to the monogenic setting. We then investigate in Section~\ref{s:spectralline}, the problem of estimating the amplitude and instantaneous frequency of a monogenic spectral line of the form (\ref{e:monogwaves}). Subsequently, we define in Section~\ref{s:synchro} the bidimensional synchrosqueezing and its application to the decomposition of multicomponent images. Finally, Section~\ref{s:numerical1} gives some insights on the practical implementation of bidimensional synchrosqueezing, and provides numerical experiments. Proofs are postponed to Section~\ref{s:proofs} to make reading easier.

\section{2D-wavelet transform and monogenic wavelet transform}\label{s:monogwav}
In the one dimensional setting, the synchrosqueezing method is based on the one--dimensional wavelet analysis: indeed, the wavelet transform is used both to decompose an analytic signal into several components, and also to estimate the amplitude and the instantaneous frequency of each component. The main advantage of using the wavelet analysis is to avoid the use of the discrete Hilbert transform, since there is a link between the wavelet coefficients of the analytic signal associated to a real function $f$, and its analytic wavelet coefficients ({\sl i.e.} the wavelet coefficients of $f$ among an analytic wavelet).

In the bidimensional case, it is then quite natural to wonder if this approach can be extended using bidimensional wavelet analysis, replacing the analytic signal by the monogenic signal associated to our data.
Since the monogenic signal can be viewed as a 3D vector field, it can be componentwise analyzed by the usual wavelet transform: Section~\ref{s:wavanal} recalls first some basics about the bidimensional continuous wavelet transform. Thereafter in Section~\ref{s:defmonogwav}, we recall the main characteristics of the {\it monogenic wavelet analysis}, which is an extension of the usual wavelet analysis introduced in~\cite{olhede:metikas:2009a} and~\cite{unser:vandeville:2009}. We would like to point out that, for any bidimensional real image $f$, the wavelet coefficients of the monogenic signal $\mathcal{M}f$ can be related to the {\it monogenic wavelet coefficients} of $f$ (see Proposition~\ref{pro:wavcoeff:monog} for a precise statement). Then, using monogenic wavelet analysis, one can recover the characteristics of the monogenic signal associated to our data, without estimating it.

In what follows, for any $(a,b,\alpha)\in \mathbb{R}^*_+\times \mathbb{R}\times (0,2\pi)$, we denote by $D_a$ the dilation operator, $T_b$ the translation operator and  $R_\alpha$ the rotation operator defined on $L^2(\mathbb{R}^2)$ by:
\[
D_a f(x)=a^{-1}f(x/a),\,T_b f(x)=f(x-b),\,R_\alpha f(x)=f(r_{\alpha}^{-1}x)\; ,
\]
where $r_{\alpha}$ is the usual $2\times 2$ rotation matrix of angle $\alpha$:
\begin{equation}\label{e:rotation}
r_\alpha=\begin{pmatrix}\cos\alpha&-\sin\alpha\\\sin\alpha&\cos\alpha\end{pmatrix}\;.
\end{equation}
\subsection{The usual bidimensional wavelet transform}\label{s:wavanal}
We briefly recall classical definitions about the bidimensional continuous wavelet transform. A (real or complex) function $\psi\in L^2(\mathbb{R}^2)$ is called an admissible wavelet if it satisfies
\footnote{The Fourier Transform of $\psi$ being defined by:
$\widehat{\psi}(\xi)=\frac{1}{2\pi} \int_{\mathbb{R}^2} \psi(x) e^{-i\xi\cdot x} dx$
}:
\begin{equation}
\label{admissibility}
{C}_{\psi}=(2\pi)^2\int_{\mathbb{R}^2}\frac{|\widehat{\psi}(\xi)|^2}{|\xi|^2}~d\xi~<~ +\infty
\end{equation}
As usual (see~\cite{Antoine:Murenzi:Vandergheynst:Twareque:2004}), the wavelet family $\{\psi_{a,\alpha,b}\}_{(a,\alpha,b)\in \mathbb{R}^*_+\times (0,2\pi)\times \mathbb{R}^2}$ is defined by $\psi_{a,\alpha,b}= T_b R_\alpha D_a \psi$. One then defines the wavelet coefficients of any $f\in L^2(\mathbb{R}^2)$ as follows:
\[
c_f(a,\alpha,b)=\int_{\mathbb{R}^{2}}f(x)~\overline{\psi_{a,\alpha,b}(x)}~\rmd x\;.
\]
If the wavelet is assumed to be isotropic, the wavelet coefficients do not depend on $\alpha$. In this case, we will denote $c_f(a,b)=c_f(a,0,b)$ and $\psi_{a,b}=\psi_{a,0,b}$.

Any square integrable function $f$ can be recovered from its wavelet coefficients using the so--called reconstruction formula:
\begin{equation}\label{e:reconstruction}
f(x)=\frac{2\pi}{C_{\psi}}\int_{b\in \mathbb{R}^2}\int_{\alpha\in (0,2\pi)}\int_{a\in (0,+\infty)}c_f(a,\alpha,b)~\psi_{a,\alpha,b}(x)~\frac{\rmd a}{a^3}\rmd \alpha~\rmd b\;,
\end{equation}
where this equality stands in $L^2(\mathbb{R}^2)$. When the wavelet is assumed to be isotropic, we get a simpler expression:
\[
f(x)=\frac{1}{C_{\psi}}\int_{b\in \mathbb{R}^2}\int_{a\in (0,+\infty)}c_f(a,b)~\psi_{a,b}(x)~\frac{\rmd a}{a^3}\rmd b\;.
\]
Other reconstruction formulas are available (see e.g.~\cite{Antoine:Murenzi:Vandergheynst:Twareque:2004}), among which the pointwise reconstruction formula, obtained by summing over scales only. In the isotropic case, it reads:
\begin{equation}\label{e:synchro:id:iso}
\,f(x)= \frac{2\pi}{\tilde C_{\psi}}\int_{0}^{+\infty}c_f(a,x)~\frac{\rmd a}{a^2}\mbox{ with }\tilde C_{\psi}= \int_{\mathbb{R}^2}\frac{\overline{\widehat{\psi}(\xi)}}{|\xi|^2}~\rmd\xi\;.
\end{equation}
As in the one dimensional case \cite{daubechies:2011}, this formula will play a special role in the bidimensional synchrosqueezing.\\
For any real function $f\in L^2(\mathbb{R}^2)$, one can also define the wavelet transform of its monogenic signal
$F=\mathcal{M}f=f+\mathcal{R}_1 f ~\rmi+\mathcal{R}_2 f~\rmj\;$, as:
\[
c_F=c_f+c_{\mathcal{R}_1 f}~\rmi+c_{\mathcal{R}_2 f}~\rmj=\begin{pmatrix}c_f\\c_{\mathcal{R}_1 f}\\c_{\mathcal{R}_2 f}\end{pmatrix}\;.
\]
\subsection{The monogenic wavelet transform}\label{s:defmonogwav}
We now present the continuous monogenic wavelet transform as defined in~\cite{olhede:metikas:2009a}. We consider a real admissible wavelet $\psi$ and define $\psi^{(M)}=\mathcal{M}\psi=\begin{pmatrix}\psi\\ \mathcal{R}_1\psi\\ \mathcal{R}_2 \psi\end{pmatrix}$ the associated monogenic wavelet. Observe that $\psi^{(M)}$ is also an admissible wavelet (taking values in $\mathbb{R}^3$), since each of its components satisfies relation (\ref{admissibility}). We define below the monogenic wavelet coefficients of a real function.
\begin{definition}
Let $f\in L^2(\mathbb{R}^2)$. The monogenic wavelet coefficients $c_f^{(M)}(a,\alpha,b)$ of $f$ are defined by:
\begin{equation}\label{e:wc}
c_f^{(M)}(a,\alpha,b)=\int_{\mathbb{R}^{2}}f(x)~{\psi^{(M)}_{a,\alpha,b}(x)}~\rmd x\;,
\end{equation}
where for any $(a,\alpha,b)\in (0,+\infty)\times (0,2\pi)\times \mathbb{R}^2$, $\psi^{(M)}_{a,\alpha,b}=T_b R_\alpha D_a(\mathcal{M}\psi)$.
\end{definition}
\begin{remark}
For any $(a,\alpha,b)$, $c_f^{(M)}(a,\alpha,b)$ is a Clifford vector, and thus can be written as $c_f^{(M)}(a,\alpha,b)=c_f(a,\alpha,b)+c_f^{(1)}(a,\alpha,b)~\rmi+c_f^{(2)}(a,\alpha,b)~\rmj$ where we denote for $i=1,2$
\[
c_f^{(i)}(a,\alpha,b)=\int_{\mathbb{R}^{2}}f(x)~(T_b R_\alpha D_a \mathcal{R}_i\psi)(x)~\rmd x\;,
\]
(see Appendix~\ref{s:appendixQuater} for more details).
\end{remark}

The next proposition states that there exists an explicit relationship between the monogenic coefficients of a real image $f$ and the usual wavelet coefficients of its monogenic signal $\mathcal{M}f$.
\begin{proposition}\label{pro:wavcoeff:monog}
Let $f\in L^2(\mathbb{R}^2)$ and denote by $F=\mathcal{M}f$ the monogenic signal associated to $f$. Then for any $(a,\alpha,b)\in (0,+\infty)\times (0,2\pi)\times \mathbb{R}^2$
\begin{equation}\label{e:wavcoeff:monog}
c_{F}(a,\alpha,b)=\begin{pmatrix}1&0\\0&-r_{\alpha}\end{pmatrix}c_f^{(M)}(a,\alpha,b)\;.
\end{equation}
\end{proposition}
\begin{proof}
We will prove that
\[
c_f^{(M)}(a,\alpha,b)=\begin{pmatrix}1&0\\0&-r_\alpha^{-1}\end{pmatrix}c_{F}(a,\alpha,b)\;,
\]
which is equivalent to (\ref{e:wavcoeff:monog}). By definition of the monogenic wavelet transform one has:
\[	c_f^{(M)}(a,\alpha,b)=\int_{\mathbb{R}^2} f(x)~{T_b R_\alpha D_a(\mathcal{M}\psi)}(x)~\rmd x\;.
\]
The translation-, scale-invariance  and steerability properties of the Riesz transform recalled in Propositions~\ref{pro:invar} and~\ref{pro:steerability} of Appendix \ref{s:appendixRiesz}, imply that:
\[
T_b R_\alpha D_a(\mathcal{R}\psi)=r_\alpha^{-1}\mathcal{R}(\psi_{a,\alpha,b})\;.
\]
Hence
\begin{equation}\label{e:1}
\int_{\mathbb{R}^2} f(x)~{T_b R_\alpha D_a (\mathcal{R}\psi)}(x)~\rmd x=r_\alpha^{-1}\int_{\mathbb{R}^2} f(x)~{\mathcal{R}(\psi_{a,\alpha,b})(x)}~\rmd x\;.
\end{equation}
Since by Proposition~\ref{e:Ri}, the Riesz transform is a componentwise antisymmetric operator on $L^2(\mathbb{R}^2)$, one deduces that:
\begin{equation}\label{e:2}
\int_{\mathbb{R}^2} f(x){\mathcal{R}(\psi_{a,\alpha,b})(x)}~\rmd x=\begin{pmatrix}< f,\mathcal{R}_1(\psi_{a,\alpha,b})>_{L^2(\mathbb{R}^2)}\\< f,\mathcal{R}_2(\psi_{a,\alpha,b})>_{L^2(\mathbb{R}^2)}\end{pmatrix}=-\int_{\mathbb{R}^2} (\mathcal{R}f)(x)~{\psi_{a,\alpha,b}(x)}~\rmd x\;.
\end{equation}
Equations~(\ref{e:1}) and (\ref{e:2}) then directly imply relation~(\ref{e:wavcoeff:monog}).
\end{proof}

We now consider the case an admissible {\it isotropic} real--valued wavelet $\psi$. Examples include the Mexican hat (Laplacian of Gaussian) or the Morse wavelets (see~\cite{olhede:metikas:2009b} and Section \ref{s:numerical1}). In such a case, the wavelets $\psi_{a,\alpha,b}$ as well as the wavelet coefficients $c_{F}(a,\alpha,b)$ of the monogenic signal do not depend on the orientation $\alpha$.  From now, we then denote $\psi_{a,\alpha,b}$ and  $c_{F}(a,\alpha,b)$ respectively by $\psi_{a,b}$ and $c_{F}(a,b)$. Thus equation (\ref{e:wavcoeff:monog}) can be simplified in the following way:
\begin{equation}\label{e:wavcoeff:monogenic:iso}
\begin{array}{lll}
c_{F}(a,b)&=&\begin{pmatrix}1&0\\0&-r_{\alpha}\end{pmatrix}c_f^{(M)}(a,\alpha,b)=\begin{pmatrix}1&0\\0&-r_{\alpha}\end{pmatrix}\begin{pmatrix}c_f(a,b)\\c^{(1)}_f(a,\alpha,b)\\c^{(2)}_f(a,\alpha,b)\end{pmatrix}\\
&=&\begin{pmatrix}c_f(a,b)\\-\cos\alpha~c^{(1)}_f(a,\alpha, b)+\sin\alpha~c^{(2)}_f(a,\alpha, b)\\-\sin\alpha~c^{(1)}_f(a,\alpha, b)-\cos\alpha~c^{(2)}_f(a,\alpha, b)\end{pmatrix}\;.
\end{array}
\end{equation}
In practical situations, it will be easier to consider the case in which $\alpha=0$. We then get that \begin{equation}\label{e:wavcoeff:monogenic:iso2}
c_{F}(a,b)=\begin{pmatrix}1&0\\0&-Id\end{pmatrix}c_f^{(M)}(a,0, b)\;.
\end{equation}
To sum up, in the isotropic case, the monogenic wavelet coefficients of the 2D real image $f$ can be related in a very simple way to the wavelet coefficients of the monogenic signal using~(\ref{e:wavcoeff:monogenic:iso2}). It will be useful in the sequel since when analyzing images, one deals with a real--valued signal $f$ and aims at recovering the wavelet coefficients $c_{F}(a,b)$ of its monogenic signal $F$ without computing $F$.
\section{Wavelet analysis of bidimensional spectral lines}\label{s:spectralline}
In this section, we tackle the problem of estimating the amplitude and instantaneous frequency of a bidimensional mode of the form $A(x)\rme^{\varphi(x)n(x)}$. In the one dimensional setting, the ``wavelet--ridge'' method has proved to be efficient to address this problem \cite{carmona:hwang:torresani:1997, carmona:hwang:torresani:1999}. In~\cite{olhede:metikas:2009a,olhede:metikas:2009b} it has been extended to the bidimensional setting using monogenic wavelet analysis.

Here we follow an alternative approach consisting in extending synchrosqueezing to the bidimensional context. We first define precisely the bidimensional modes that we are considering, and that we will call ``Intrinsic Monogenic Mode Function'' (IMMF). Then we deduce the notion of bidimensional frequency vector and state an approximation result using the monogenic wavelet transform.

Let us first define our concept of bidimensional modes:
\begin{definition}
\label{def:IMMF}
Let $\varepsilon>0$. An Intrinsic Monogenic Mode Function (IMMF) with accuracy $\varepsilon$ and smoothness $\sigma>0$ is a function $F\in \dot{B}^{-\sigma}_{\infty,1}(\mathbb{R}^2,\mathbb{H})$ of the form
\begin{equation}\label{e:IMMF}
F(x)=A(x)\rme^{\varphi(x)n_{\theta(x)}}\mbox{ with }n_{\theta(x)}=\cos(\theta(x))\rmi+\sin(\theta(x))\rmj\;,
\end{equation}
where $A>0$ and $A,\theta\in\mathcal{C}^1(\mathbb{R}^2)\cap L^\infty(\mathbb{R}^2)$, $\varphi\in \mathcal{C}^2(\mathbb{R}^2)$. The function $A$ is called the amplitude of $F$, whereas $\varphi$ and $n_\theta$ are called respectively the scalar phase and the local orientation of $F$.\\
The function $\varphi$ is assumed to have bounded derivatives of order $1$ and $2$:
\begin{equation}\label{e:hyp1a}
\forall i\in\{1,2\},\;0<\inf_{x\in\mathbb{R}^2}\left|\partial_{x_i} \varphi(x)\right|\leq \sup_{x\in\mathbb{R}^2}\left|\partial_{x_i} \varphi(x)\right|<\infty\;.
\end{equation}
\begin{equation}\label{e:M}
M=\max_{i_1,i
_2=1,2}\sup_{x\in \mathbb{R}^2} \left|\partial^2_{x_{i_1} x_{i_2}}\varphi(x)\right|<\infty\;.
\end{equation}
Further we assume that $A,\theta,\nabla^2\varphi$ are slowly varying functions with respect to $\nabla\varphi$, namely that for $i=1,2$
\begin{equation}\label{e:hyp3a}
\forall x\in\mathbb{R}^2, \max\left(|\partial_{x_i}A(x)|,|\partial_{x_i} \theta(x)|\right)\leq \varepsilon|\partial_{x_i}\varphi(x)|\;.
\end{equation}
and
\begin{equation}\label{e:hyp4a}
\forall x\in\mathbb{R}^2, \max_{i_1,i_2=1,2}\left|\partial^2_{x_{i_1} x_{i_2}}\varphi(x)\right|\leq \varepsilon|\nabla \varphi(x)|\;.
\end{equation}
\end{definition}
\begin{remark}\label{rem:besov}
Observe that, in the sequel, we shall consider non--square integrable functions. In this case, the two reconstruction formulas~(\ref{e:reconstruction}) and~(\ref{e:synchro:id:iso}) may not hold (see Appendix B of~\cite{jaffard:meyer:ryan:2001}). We then need an additive assumption to ensure that the integral
$
\int_{0}^{+\infty}c_{F}(a,b)\frac{\rmd a}{a^{2}}\;
$
exists to give some sense to the reconstruction formula~(\ref{e:synchro:id:iso}). Moreover, we will estimate the behavior of
$$\int_{a_0}^{+\infty}|c_{F}(a,b)| \frac{\rmd a}{a^{2}}$$ for any $a_0>0$.\\
The additional hypothesis $F\in \dot{B}^{-\sigma}_{\infty,1}(\mathbb{R}^2,\mathbb{H})$ for $\sigma>0$ is equivalent to the following condition on the wavelet coefficients (see for e.g. \cite{triebel:1978}):
\begin{equation}
\label{besov}
\|F\|_{\dot{B}^{-\sigma}_{\infty,1}}=\int_{0}^\infty \sup_{b\in\mathbb{R}^2}|c_F(a,b)|\frac{\rmd a}{a^{2-\sigma}}<+\infty\;.
\end{equation}
In such case, we easily deduce that for any $a_0>0$:
\begin{equation}\label{e:remc}
\sup_{b\in\mathbb{R}^2}\int_{a_0}^\infty |c_F(a,b)|\frac{\rmd a}{a^{2}} = \sup_{b\in\mathbb{R}^2}\int_{a_0}^\infty |c_F(a,b)| \frac{1}{a^{\sigma}}\frac{\rmd a}{a^{2-\sigma}}\leq a_0^{-\sigma}~\|F\|_{\dot{B}^{-\sigma}_{\infty,1}}
\end{equation}
Moreover, by hypothesis, $F\in \mathcal{C}^1(\mathbb{R}^2,\mathbb{H})$ and $\nabla F$ is bounded. Using the expression of the wavelet coefficients then implies that:
$$
|c_F(a,b)| \leq a^2 \|\nabla F\|_{L^{\infty}} \|x \psi\|_{L^1}~,~~~~\forall b\in\R^2
$$
which ensures that for any $a_0>0$
\begin{equation}\label{e:remb--}
\sup_{b\in\mathbb{R}^2}\int_0^{a_0}|c_F(a,b)|\frac{\rmd a}{a^2}<\infty\;.
\end{equation}
and finally leads to the existence of the integral
$\int_{0}^{+\infty}c_{F}(a,b)\frac{\rmd a}{a^{2}}$.
\end{remark}
\begin{remark}
In~\cite{olhede:metikas:2009b} authors considered spectral lines of the form $f(x)=A(x)\cos(\varphi(x))$ with $A,\nabla\varphi$ slowly varying with respect to $\varphi$. The authors then approximate the monogenic signal $F$ associated to $f$ by an IMMF of the form~(\ref{e:IMMF}) with  constant amplitude,  linear phase and  constant local orientation.
\end{remark}
We now focus on the estimation of the amplitude and instantaneous frequency of an IMMF. To this end, we need some assumptions on the wavelet that we list below:\\
\noindent{\bf Assumptions (W)}
\begin{enumerate}[(i)]
\item the wavelet $\psi$ is isotropic and real--valued.
\item $\psi\in W^{1,1}(\mathbb{R}^2)$.
\item The first three moments of $\psi$ and $\nabla\psi$ are finite, namely
\[
\sup_{\alpha\in \{1,2,3\}}I_\alpha<\infty,\,\sup_{\alpha\in \{1,2,3\}}I'_\alpha<\infty\;,
\]
where for any $\alpha>0$
\begin{equation}\label{integrales}
I_\alpha=\int_{\mathbb{R}^2} |x|^\alpha |\psi(x)|\rmd x~,~~\,I'_\alpha=\int_{\mathbb{R}^2}|x|^\alpha|\nabla\psi(x)|\rmd x\;.
\end{equation}
\end{enumerate}
We first deal with the simple case of an IMMF of constant amplitude and phase:
\begin{lemma}\label{lem:cst}
Let $F$ be an IMMF of the form $F(x)=A_0\rme^{(k\cdot x) n_{\theta}}$ with $k=(k_1,k_2)\in \mathbb{R}^2$, $A_0\in \mathbb{R}^*_+, \theta \in \mathbb{R}$, and let $\psi$ be an isotropic real wavelet belonging to $W^{1,1}(\mathbb{R}^2)$. \\
The wavelet transform of $F$ is given by:
\begin{equation}\label{e:wavcoeff0}
c_{F}(a,b)=A_0 a\widehat{\psi}(a k)\left(\cos(k\cdot b)+\sin(k\cdot b)(\cos\theta~\rmi+\sin\theta~ \rmj)\right)=a\widehat{\psi}(a k)\left(A_0 \rme^{(k\cdot b)n_{\theta}}\right)\,
\end{equation}
and one has for $i=1,2$:
\begin{equation}\label{e:wavcoeffdiff}
\partial_{b_i}c_{F}(a,b)=k_i n_{\theta}\left(a\widehat{\psi}(a k)\right)\left(A_0\rme^{(k\cdot b)n_{\theta}}\right)\;.
\end{equation}
The vectors $k$ and $n_{\theta}$ can then be recovered using the two following relations:
\[
k_1 n_{\theta} =\partial_{b_1} c_{F}(a,b)\times (c_{F}(a,b))^{-1}\;,
\]
\[
k_2 n_{\theta} =\partial_{b_2} c_{F}(a,b)\times (c_{F}(a,b))^{-1}\;.
\]
\end{lemma}
\begin{proof}
By definition, since the wavelet is isotropic real, one has:
\[
c_F(a,b)=a^{-1}\int_{\mathbb{R}^2}F(x)\psi\left(\frac{x-b}{a}\right)\rmd x
= a\int_{\mathbb{R}^2}F(a u +b)\psi(u)\rmd u
\]
A simple calculation leads to:
\[
c_F(a,b)=a A_0
\begin{pmatrix}\int_{\mathbb{R}^2}\cos(k\cdot (au+b)) \psi(u)\rmd u\\
\cos \theta \int_{\mathbb{R}^2}\sin(k\cdot (au+b))\psi(u)\rmd u\\
\sin\theta\int_{\mathbb{R}^2}\sin(k\cdot (au+b))\psi(u)\rmd u
\end{pmatrix}
=A_0 a\widehat{\psi}(a k)
\begin{pmatrix}\cos(k.b)\\\sin(k.b)\cos(\theta)\\\sin(k.b)\sin(\theta)
\end{pmatrix}=a\widehat{\psi}(a k)F(b)\;.
\]
where we have used $\widehat{\psi}(-a k)=\widehat{\psi}(a k)$ since the wavelet is isotropic.

\

Now, since $\psi\in W^{1,1}(\mathbb{R}^2)$ and $F\in L^{\infty}(\mathbb{R}^2)$, we get:
$$
\partial_{b_i}c_{F}(a,b)=a^{-1}\int_{\mathbb{R}^2}F(x)\partial_{b_i}\left[\psi\left(\frac{x-b}{a}\right)\right]\rmd x
=-a^{-2}\int_{\mathbb{R}^2}F(x)(\partial_{x_i}\psi)\left(\frac{x-b}{a}\right)\rmd x\;.
$$
A similar calculation, using  $\widehat{\partial_{x_i}\psi}(\xi)=\rm i \xi_i \widehat{\psi}(\xi), \forall \xi=(\xi_1,\xi_2)$, and $\psi$ isotropic, yields:
\[
\partial_{b_i}c_F(a,b)
=A_0  a k_i \widehat{\psi}(a k)
\begin{pmatrix}- \sin(k.b)\\ \cos(k.b)\cos(\theta)\\ \cos(k.b)\sin(\theta)
\end{pmatrix}=a k_i \widehat{\psi}(a k)~n_\theta F(b)\;,
\]
where we have used:
$$n_\theta F(b)=n_\theta (\cos(k\cdot b)+n_\theta \sin(k\cdot b))=- \sin(k\cdot b) +n_\theta \cos(k\cdot b).$$
This leads to (\ref{e:wavcoeffdiff}).
\end{proof}
Remark that, more generally, if  $F(x)=A_0 \rme^{(k\cdot x+\alpha) n_{\theta}}$, ($\alpha \in \R$):
\begin{eqnarray}
\label{ext}
c_{F}(a,b)&=&A_0 a\widehat{\psi}(a k)\rme^{ (k\cdot b+\alpha) n_{\theta}}=a  \widehat{\psi}(a k) ~F (b)\\
\partial_{b_i}c_F(a,b)&=&a k_i \widehat{\psi}(a k)~n_\theta F(b)
\end{eqnarray}
The formulas~(\ref{e:wavcoeff0}) and (\ref{e:wavcoeffdiff}) can be extended in the general case of an IMMF $F$ of the form~(\ref{e:IMMF}).
\begin{proposition}\label{pro:wavcoeffIMMF}
Let $\psi$ a wavelet satisfying assumptions (W) and $F$ an IMMF with accuracy $\varepsilon$ and smoothness $\sigma>0$ of the form~(\ref{e:IMMF}). For all $(a,b)\in\mathbb{R}^*_+\times\mathbb{R}^2$, one has
\begin{enumerate}[(i)]
\item\label{pro:wavcoeffIMMFi}
\[
c_{F}(a,b)= a\widehat{\psi}(a\nabla \varphi(b))\left(A(b)\rme^{\varphi(b)n(b)}\right)+
 \varepsilon a^2 R_1(a,b)\;,
\]
with
\begin{equation}\label{e:remaind1}
\begin{array}{lll}
|R_1(a,b)| &\leq& I_1  (\sqrt{2}A(b)+1)|\nabla\varphi(b)|+a I_2 A(b)\left(\left|\nabla \varphi(b)\right|+  \sqrt{2}M\right)/2\\
&&+a I_2 M/2+a^{2} I_3 A(b)M/6\;.
\end{array}
\end{equation}
\item\label{pro:wavcoeffIMMFii}
\[
\partial_{b_i} c_F(a,b)=\partial_{b_i} \varphi(b) n_{\theta(b)}\left(a\widehat{\psi}(a\nabla \varphi(b))\right)\left(A(b)\rme^{\varphi(b)n_{\theta(b)}}\right)+\varepsilon a R_2(a,b)\;,
\]
with
\begin{equation}\label{e:remaind2}
\begin{array}{lll}
|R_2(a,b)|&\leq& A(b)\left(a\left|\nabla \varphi(b)\right|I'_2/2+a^2 M I'_3/6\right)\\
&&+ \left(\sqrt{2}A(b)+1\right)\left(|\nabla\varphi(b)|I'_1+ a M I'_2/2\right)\;.
\end{array}
\end{equation}
\end{enumerate}
\end{proposition}
\begin{remark}
As a consequence of this proposition, since in the isotropic case $c_{F}(a,b)$ and $c_f^{(M)}(a,\alpha,b)$ are related by~(\ref{e:wavcoeff:monogenic:iso}), the above proposition specifies the results stated in~\cite{olhede:metikas:2009a} and~\cite{olhede:metikas:2009b}. Notably, it gives a bound of the approximation error.
\end{remark}
\begin{proof}
See Section~\ref{s:proofs11}.
\end{proof}
We now define the two following Clifford vectors:
\begin{eqnarray}\label{e:L1}
\Lambda_1(a,b) & = & ~ \partial_{b_1} c_{F}(a,b)\times (c_{F}(a,b))^{-1}\;,\\
\nonumber
\Lambda_2(a,b) &= &~\partial_{b_2} c_{F}(a,b)\times (c_{F}(a,b))^{-1}\;.
\end{eqnarray}
These two vectors will provide an approximation of the vectors $\partial_{b_i}\varphi(b) n_{\theta(b)}$ for an IMMF $F$ of the form~(\ref{e:IMMF}). More precisely, we can state the following result:
\begin{theorem}\label{th:main1}
Let $F$ be an IMMF with accuracy $\varepsilon>0$ and smoothness $\sigma>0$ of the form~(\ref{e:IMMF}).
Assume that we are given $\psi$ a wavelet satisfying assumptions (W) and consider $(a,b)\in \R^*_+\times \mathbb{R}^2$ such that

\[
|c_F(a,b)|\geq \varepsilon^\nu\mbox{ for some }\nu\in (0,1/2)\;.
\]
Then for $i=1,2$:
\begin{equation}\label{e:thmain1}
|\Lambda_i(a,b)-\partial_{b_i}\varphi(b)n_{\theta(b)}|\leq \varepsilon^{1-\nu}a~\left(|R_2(a,b)|+ a |R_1(a,b)|~|\nabla\varphi(b)|\right)
\end{equation}
where $R_1$, $R_2$ have been defined respectively in~(\ref{e:remaind1}) and (\ref{e:remaind2}).
In addition for any $a_0>0$, there exists some $\varepsilon_0>0$ such that, for all $(a,b)\in (0,a_0)\times \mathbb{R}^2$ and any $0<\varepsilon\leq \varepsilon_0$:
 \begin{equation}\label{e:thmain1bis}
|\Lambda_i(a,b)-\partial_{b_i}\varphi(b)n_{\theta(b)}|\leq \varepsilon^{\nu}\end{equation}
\end{theorem}
\begin{proof}
See Section~\ref{s:proofs12}.
\end{proof}
Since $\cos(\varphi(x))=\cos(-\varphi(x))$, there is an ambiguity in the definition of the instantaneous frequency $\nabla\varphi$ since $\varphi$ can be replaced by $-\varphi$. To solve this problem, practitioners usually assume that the first component of the instantaneous frequency is positive (see~\cite{murray:2012} and references therein for more details). Under this additional assumption, our method yields an estimate of the amplitude and of the instantaneous frequency
of an IMMF:
\begin{proposition}\label{pro:main1}
The notations and assumptions are those of Theorem~\ref{th:main1}. In addition, we assume that:
\[
\forall x\in\mathbb{R}^2,\,\partial_{x_1}\varphi(x)>0\;.
\]
Then for $i=1,2$, and $(a,b)\in (0,a_0)\times\mathbb{R}$:
\begin{equation}\label{e:promain1b}
|\partial_{b_1}\varphi(b)-|\Lambda_1(a,b)||\leq \varepsilon^{\nu}\;.
\end{equation}
and
\begin{equation}\label{e:promain1c}
|\partial_{b_2}\varphi(b)-|\Lambda_2(a,b)|~\mathrm{sgn}(\mathrm{Re}(\partial_{b_1} c_F(a,b)~\overline{\partial_{b_2} c_F(a,b)}))|\leq \varepsilon^{\nu}\;.
\end{equation}
\end{proposition}
\begin{proof}
See Section~\ref{s:proofs2}.
\end{proof}
\section{Decomposition of a multi--component image into spectral lines }\label{s:synchro}
In this section we will consider a more general context, where our model is now a superposition of several IMMFs (introduced in Section~\ref{s:spectralline}), assumed to be well separated, as defined in the following definition \ref{def:supIMMF}.
We will show in Theorem~\ref{th:main2} how such functions can be decomposed into monogenic modes, using synchrosqueezing. Thereafter, for each component, the associated amplitude and instantaneous frequency will be estimated using the results of Section~\ref{s:spectralline}.
\begin{definition}\label{def:supIMMF}
A function $F$ defined from $\mathbb{R}^2$ to $\mathbb{H}$ is said to be a superposition of well separated Intrinsic Monogenic Mode Components with smoothness $\sigma>0$ up to accuracy $\varepsilon>0$ and with separation $d>0$, if there exists a finite integer $L$ such that
\begin{equation}\label{e:supIMMF}
F(x)=\sum_{\ell=1}^L F_\ell(x)\;,
\end{equation}
where all the $F_\ell$ are IMMFs with accuracy $\varepsilon$ and smoothness $\sigma_\ell\geq \sigma>0$ of the form (\ref{e:IMMF}): $F_\ell(x)=A_{\ell}(x)~\rme^{\varphi_{\ell}(x)n_{\theta_{\ell}(x)}}$, and moreover satisfy for any $x$, $\ell>\ell'$ and $i\in\{1,2\}$:
\begin{equation}\label{e:sep}
\left|\partial_{x_i}\varphi_\ell(x)\right|>\left|\partial_{x_i}\varphi_{\ell'}(x)\right|\;,
\end{equation}
and
\begin{equation}\label{e:sepbis}
\left|\partial_{x_i}\varphi_\ell(x)n_{\theta_\ell(x)}-\partial_{x_i}\varphi_{\ell'}(x)n_{\theta_{\ell'}(x)}\right|\geq d~\left[ \left|\partial_{x_i}\varphi_{\ell}(x)\right|+\left|\partial_{x_i}\varphi_{\ell'}(x)\right|\right]\;.
\end{equation}
We denote by $\mathcal{A}_{\varepsilon,\sigma,d}$ the class of all functions $F$ satisfying these conditions.
\end{definition}
We now define the monogenic synchrosqueezing transform (MSST) of a function $F$ belonging to $\mathcal{A}_{\varepsilon,\sigma,d}$.
\begin{definition}\label{d:SST}
Assume that we are given $\psi$ a wavelet satisfying assumptions (W) such that
\[
\mathrm{supp}(\widehat{\psi})\subset \{\xi\in\mathbb{R}^2~;~1-\Delta\leq|\xi|\leq 1+\Delta\}\;,
\]
with
\begin{equation}\label{e:Delta}
\Delta< \frac{d}{2(1+d)}\;.
\end{equation}
Consider $h\in\mathcal{C}^\infty_c(\mathbb{R},\mathbb{R})$ such that $\int_{\mathbb{R}} h(x)\rmd x=1$. Let $\delta>0$, $\nu\in (0,1)$, and $\varepsilon >0$. For $F\in \mathcal{A}_{\varepsilon,\sigma,d}$, the monogenic synchrosqueezing transform of $F$ is defined for any $(b,k,n)\in \mathbb{R}^2\times \mathbb{R}^2\times \mathbb{S}^1$ ($ \mathbb{S}^1$ being the unit sphere of $\R^2$) by
\begin{equation}
\label{e:defmsst}
S_{F,\varepsilon}^{\delta,\nu}(b,k,n)=\int_{A_{\varepsilon,F}(b)} c_{F}(a,b)\frac{1}{\delta^2}h\left(\frac{k_1-{\rm Re} (\overline{n}~\Lambda_1(a,b))}{\delta}\right)h\left(\frac{k_2 -{\rm Re} (\overline{n}~\Lambda_2(a,b))}{\delta}\right)\frac{\rmd a}{a^{2}}\;,
\end{equation}
with
\[
A_{\varepsilon,F}(b)=\{a\in \R_+~;\,|c_F(a,b)|>\varepsilon^{\nu}\}\;,
\]
and where $\Lambda_1(a,b)$, $\Lambda_2(a,b)$ are defined by equation~(\ref{e:L1}).
\end{definition}
We can now state our main result:
\begin{theorem}\label{th:main2}
The notations are those of Definition~\ref{d:SST}. Let $F\in\mathcal{A}_{\varepsilon,\sigma,d}$ and $\nu\in (0,1/(2+4/\sigma))$. Provided that $\varepsilon>0$ is sufficiently small, the following assertions hold:
\begin{itemize}
\item $|c_F(a,b)|>\varepsilon^{\nu}$ only if there exists some $\ell$ such that
\begin{equation}\label{e:Zl}
(a,b)\in Z_\ell=\left\{(a,b),\, \left|a|\nabla\varphi_\ell(b)|-1\right|< \Delta\right\}\;.
\end{equation}
\item  For each $\ell=\{1,\cdots, L\}$ and each pair $(a,b)\in Z_\ell$ for which $|c_F(a,b)|> \varepsilon^\nu$, one has for $i=1,2$ and for some $C>0$:
\[
|\Lambda_i(a,b)-\partial_i\varphi_\ell(b)n_{\theta_{\ell}(b)}|\leq C\varepsilon^{\nu}\;.
\]
\item Moreover for any $\ell\in\{1,\cdots,L\}$, there exists $C>0$, such that for any $b\in\mathbb{R}^2$
\[
\lim_{\delta\rightarrow 0}\left|\frac{1}{2\pi\tilde C_\psi}\int_{\mathbb{S}^1}\int_{{\cal B}_{\ell}(\varepsilon^{\nu},n,b)} S^\delta_{f,\varepsilon}(b,k,n)\rmd k\rmd n-A_\ell(b)\rme^{\varphi_\ell(b) n_{\theta_\ell(b)}}\right|\leq C\varepsilon^{\nu}\;.
\]
where we denote ${\cal B}_{\ell}(\varepsilon^{\nu},n,b)=\{k\in\R^2~;\,\max_i|k_i n-\partial_{b_i}\varphi_\ell(b) n_{\theta_{\ell}(b)}|\leq \varepsilon^{\nu}\}$, and where $\tilde C_\psi$ was introduced in~(\ref{e:synchro:id:iso}).
\end{itemize}
\end{theorem}
\begin{proof}
See Section~\ref{s:proofs:th:main2}.
\end{proof}

\section{Implementation and numerical experiments}\label{s:numerical1}

This section aims to illustrate the efficiency of the monogenic synchrosqueezing transform, and its potential applications.
We will first describe the computation of the discrete MSST from a $N\times N$ discrete image, then we will show some examples on artificial and real images.

\subsection{The discrete monogenic synchrosqueezing transform}
The first step consists in the discretization of the variables, therefore one builds discrete sets $\mathcal{A}$, $\mathcal{K}$, $\mathcal{O}$ for the scales $a$, the normalized frequencies $(k_1,k_2)$, and the orientations $\theta$ respectively, the frequency resolution in $k$ being denoted by $\Delta_k$. Note that one usually chooses a logarithmic scale for $a$ and $k$, e.g. $\mathcal{A}= \{2^{j/n_v}\}_{j=0\cdots n_a-1},$ so that we have $n_v$ coefficients per octave.
Then, we compute the discrete monogenic wavelet transform $c_F(a,b)$,
and the estimations of the instantaneous frequencies $\Lambda_1(a,b)$ and $\Lambda_2(a,b)$ defined in equation (\ref{e:L1})
by writing for $i=1,2$:
$$\partial_{b_i} c_f(a,b) = \int_{\mathbb{R}^2} f(x) \partial_{b_i} \psi_{a,b}(x)~ \rmd x ,$$
the other components $\partial_{b_i} c_{\mathcal{R}_1f}(a,b)$ and $\partial_{b_i} c_{\mathcal{R}_2f}(a,b)$ being computed in the same manner.
The monogenic synchrosqueezing transform consists in a partial reallocation of the monogenic wavelet transform according to space, frequency and orientation parameters. Given a threshold $\gamma$, and for all  $a \in \mathcal{A},(k_1,k_2)\in\mathcal{K}^2, \theta \in \mathcal{O}$, one simply writes
\begin{equation}
\label{e:dismsst}
S_{F,\gamma}(b,k_{1},k_{2},n_\theta) = \frac{\log(2)}{n_v} \sum_{a\in\mathcal{A} \mbox{ s.t. } \left\{
\begin{array}{l}
|c_F(a,b)|>\gamma \\
\left| k_1-{\rm Re} (\Lambda_1(a,b)~\overline{n}_{\theta}) \right| \le \frac{\Delta_{k_1}}{2} \\
\left| k_2-{\rm Re} (\Lambda_2(a,b)~\overline{n}_{\theta}) \right| \le \frac{\Delta_{k_2}}{2}
\end{array}
\right.
}
\frac{c_F(a,b) }{a}.
\end{equation}
The term $\frac{\log(2)}{a n_v}$ comes from the measure $\frac{da}{a^2}$ in equation (\ref{e:defmsst}), as we deal with logarithmic scales.
Note that if we use a logarithmic scale for the frequencies like in \cite{daubechies:2011}, the resolution $\Delta_k$ depends on the value $k$.

This discrete MSST is relatively easy and cheap to compute, but it is not of great interest in practice because of the large number of variables which prevents both easy representation and fast processing (i.e. decomposition, demodulation). To decrease this number of variables, one can remark that if a superposition of IMMFs satisfies the separation condition of equation (\ref{e:sep}), then for any $\ell>\ell'$ it also satisfies $|\nabla \varphi_{\ell}(x) | > |\nabla \varphi_{\ell'}(x) | ,$ so that the IMMFs are also separated in terms of the norm of the frequency vector.
This leads us to define the following discrete isotropic MSST for all $a \in \mathcal{A}$ and $k\in\mathcal{K}$:
\begin{equation}
\label{e:dismsst}
S_{F,\gamma}(b,k) = \frac{\log(2)}{n_v} \sum_{a\in\mathcal{A} \mbox{ s.t. } \left\{
\begin{array}{l}
|c_F(a,b)|>\gamma \\
\left| k- \sqrt{|\Lambda_1(a,b)|^2 + |\Lambda_2(a,b)|^2} \right| \le \frac{\Delta_k}{2} \\
\end{array}
\right.
}
\frac{c_F(a,b) }{a}.
\end{equation}
Intuitively, $S_{F,\gamma}(b,k_p)$ contains all the coefficients $c_F(a,b)$ whose ``instantaneous isotropic frequency`` $ \sqrt{|\Lambda_1(a,b)|^2 + |\Lambda_2(a,b)|^2}$ is around the value $k$.
The latter isotropic MSST can be particularly useful when visualizing the time-frequency distribution, as it only depends on three scalar parameters $(b_1,b_2,k)$, that is by far easier to represent.
Hence, this version will be used in the experiments hereafter. Additionally we will use the Morlet wavelet with parameters $\mu>0$ and $\sigma>0$, defined in the Fourier domain as follows
\begin{equation}\label{e:morlet}
\hat\psi_{\mu,\sigma}(\xi)=\exp(-\pi^2\sigma(|\xi|-\mu)^2)\;.
\end{equation}
\subsection{Representing a synthetic 3-components signal}
This section illustrates the interest in using the MSST to represent multicomponent signals in 2 dimensions. Let us first define the synthetic 3-components test-signal $f = f_1 + f_2 + f_3$ with
\begin{equation}
\label{e:testsig}
\left\{
\begin{array}{lll}
f_1(x,y) &=& e^{-10((x-0.5)^2+(y-0.5).^2))} \sin(10\pi(x^2+y^2+2(x+0.2y)) \\
f_2(x,y) &=& 1.2 \sin(40\pi(x+y))\\
f_3(x,y) &=&  \cos(2\pi(70x+20x^2+50y-20y^2-41xy))
\end{array}
\right.
\end{equation}
We compute both its monogenic isotropic wavelet transform $c_F$ and its isotropic MSST $S_F$. We then aim at representing $|c_F|$ (resp. $|S_F|$), which depends on the three scalar parameters $x$, $y$ and $a$ (resp. $k$). We will here propose two distinct visualizations: either a 3-dimensional scatter ``density'' plot, or a 2-dimensional representation for $x$ or $y$ fixed, which looks similar to a 1D continuous wavelet representation. The following Figure \ref{fig:visu} shows these different representations for the wavelet and the MSST of the signal (\ref{e:testsig}). We see how the MSST sharpens the time-scale representation, making it more readable.

\begin{figure}[H]
  \begin{minipage}[b]{0.24\linewidth}
    \centerline{
    \includegraphics[width=1.0\linewidth]{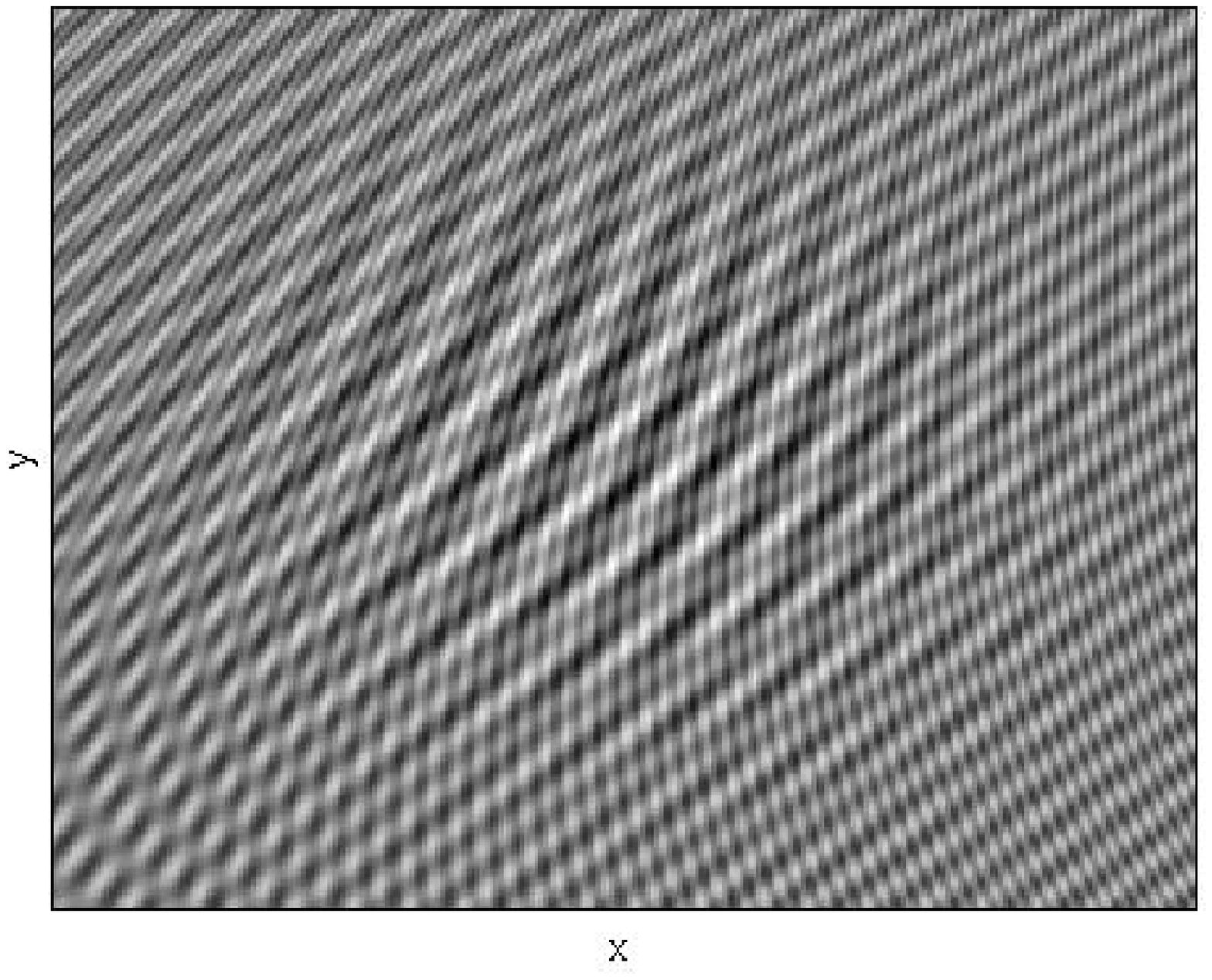}}
    \centerline{(a)}
  \end{minipage}
  \begin{minipage}[b]{0.24\linewidth}
    \centerline{
    \includegraphics[width=1.0\linewidth]{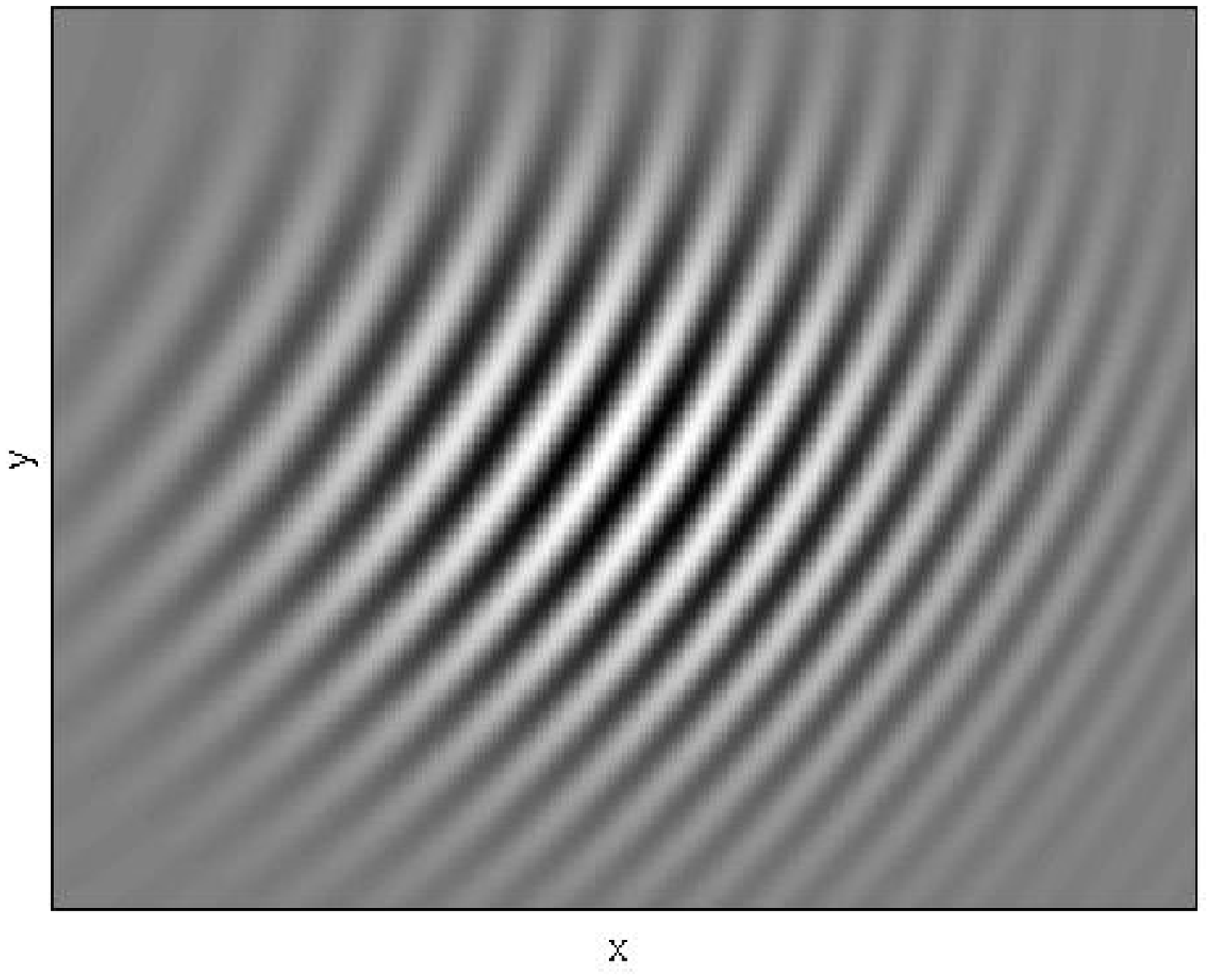}}
    \centerline{(b)}
  \end{minipage}
  \begin{minipage}[b]{0.24\linewidth}
    \centerline{
    \includegraphics[width=1.0\linewidth]{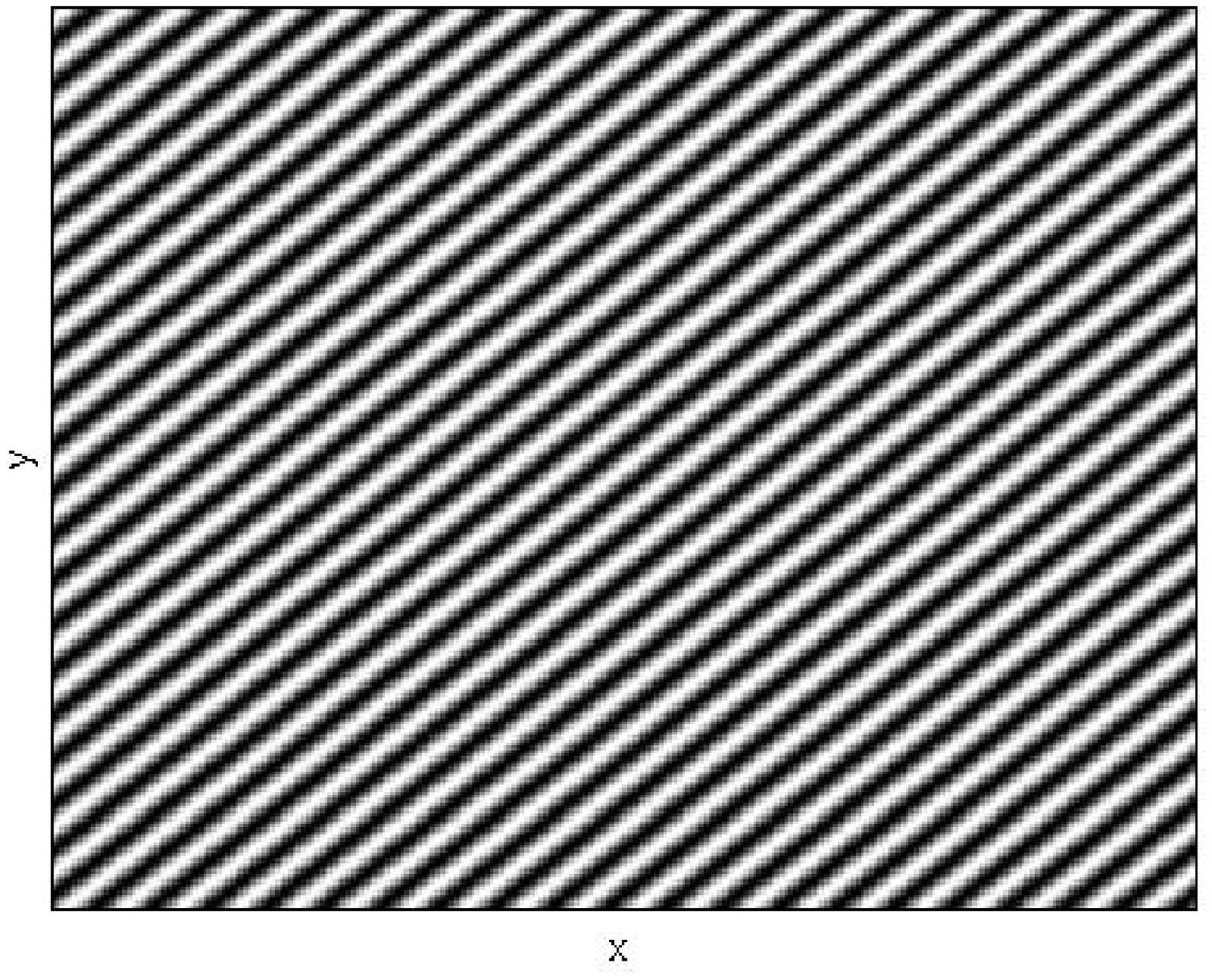}}
    \centerline{(c)}
  \end{minipage}
  \begin{minipage}[b]{0.24\linewidth}
    \centerline{
    \includegraphics[width=1.0\linewidth]{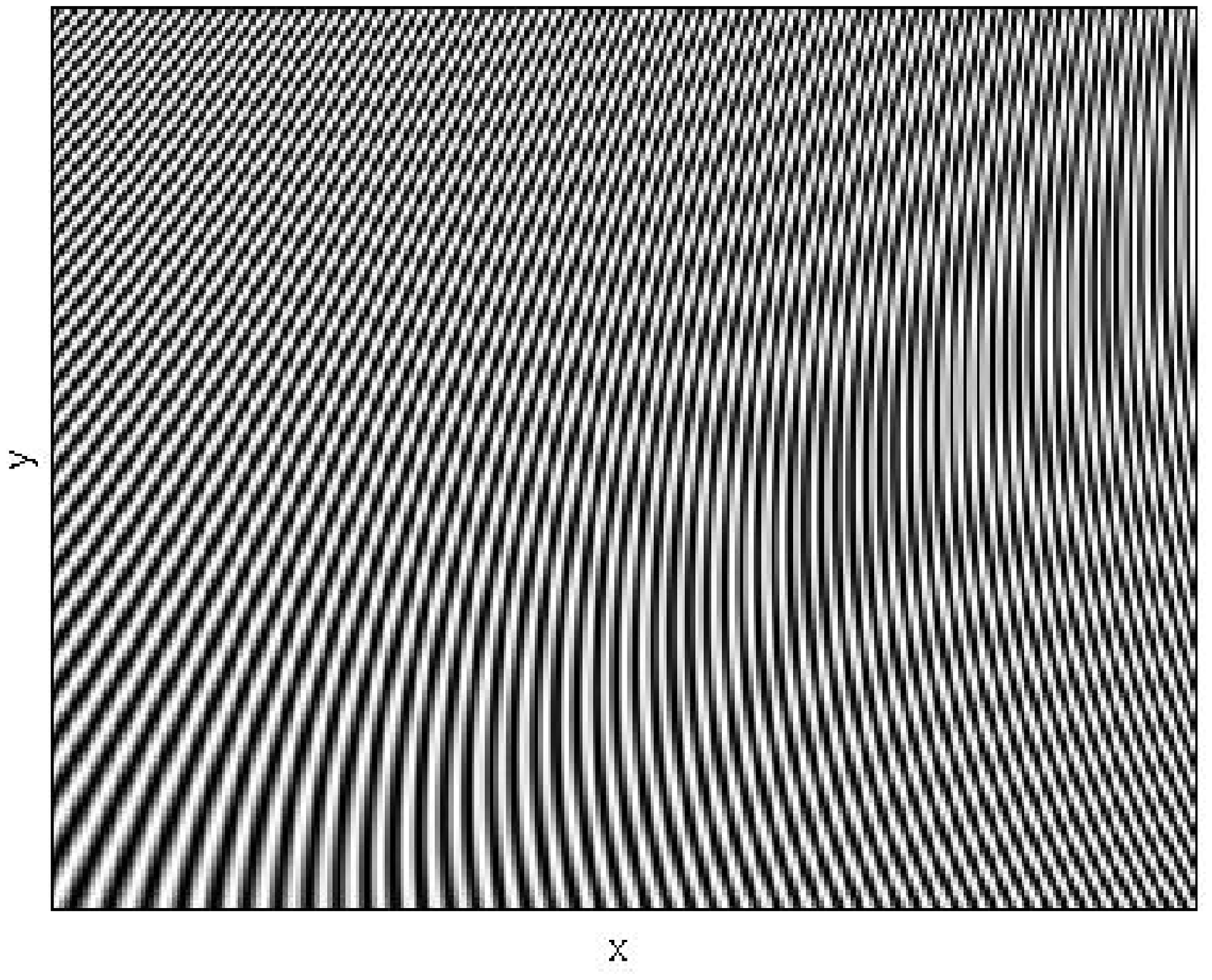}}
    \centerline{(d)}
  \end{minipage}
  \begin{minipage}[b]{0.24\linewidth}
    \centerline{
    \includegraphics[width=1.0\linewidth]{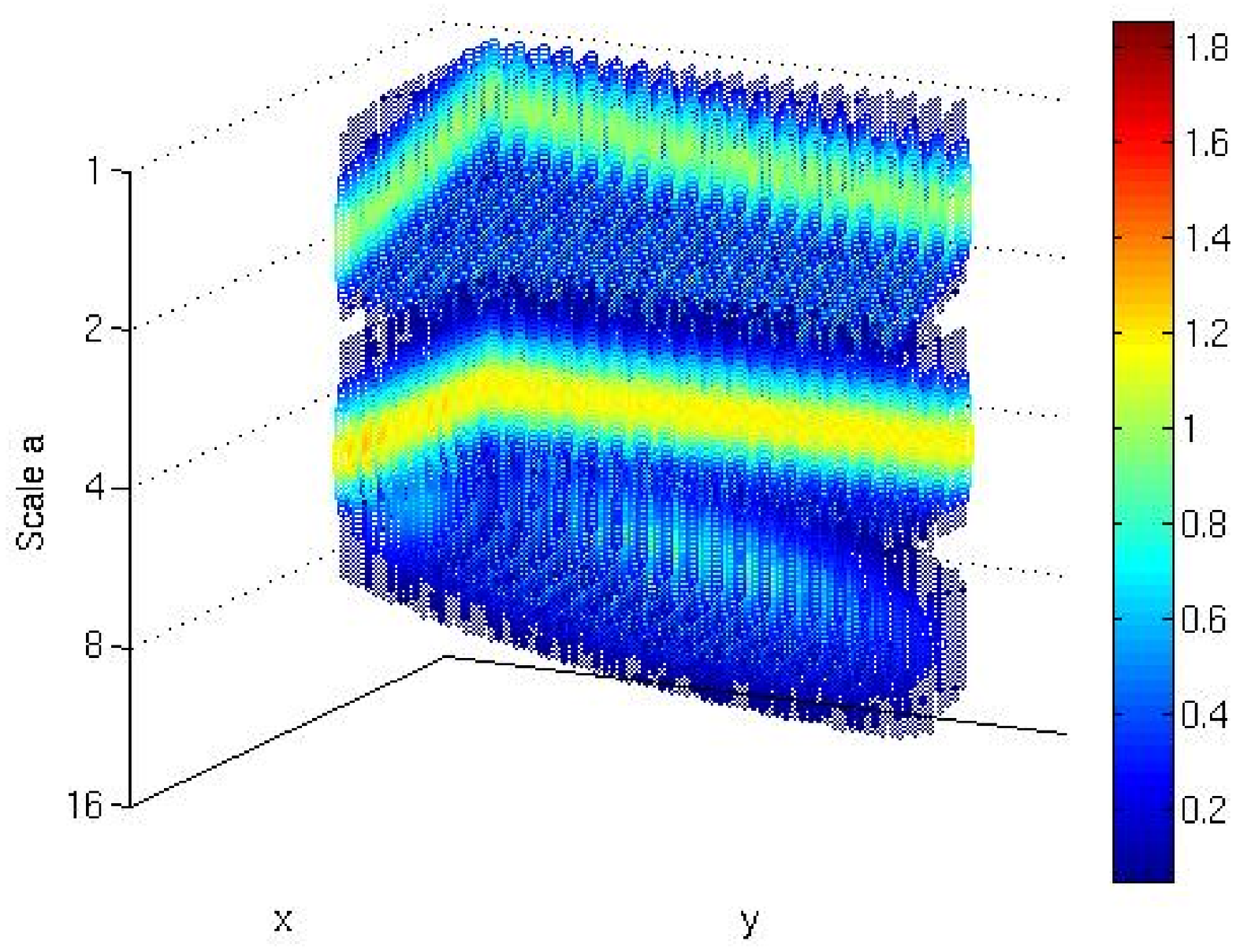}}
    \centerline{(e)}
  \end{minipage}
  \begin{minipage}[b]{0.24\linewidth}
    \centerline{
    \includegraphics[width=1.0\linewidth]{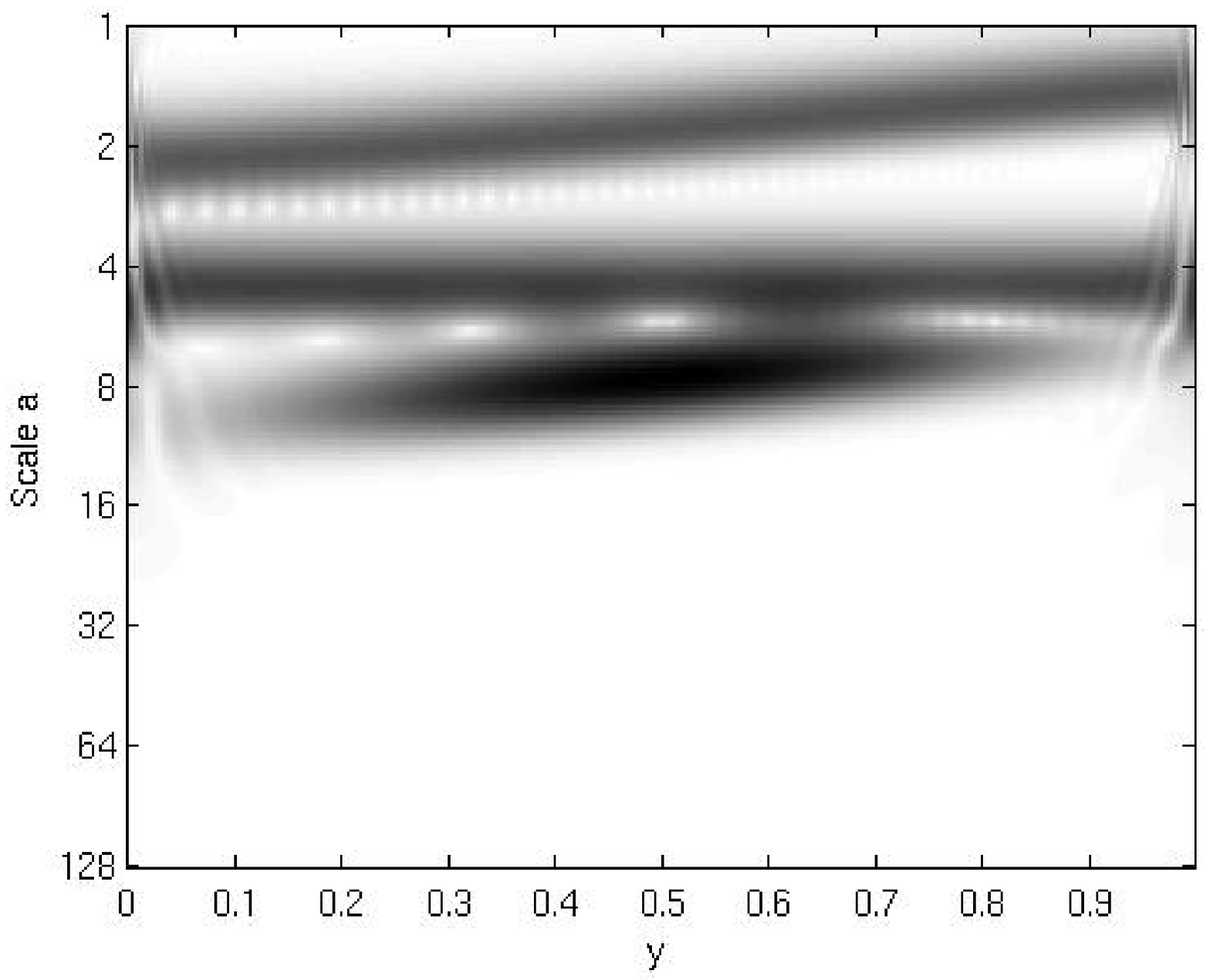}}
    \centerline{(f)}
  \end{minipage}
  \begin{minipage}[b]{0.24\linewidth}
    \centerline{
    \includegraphics[width=1.0\linewidth]{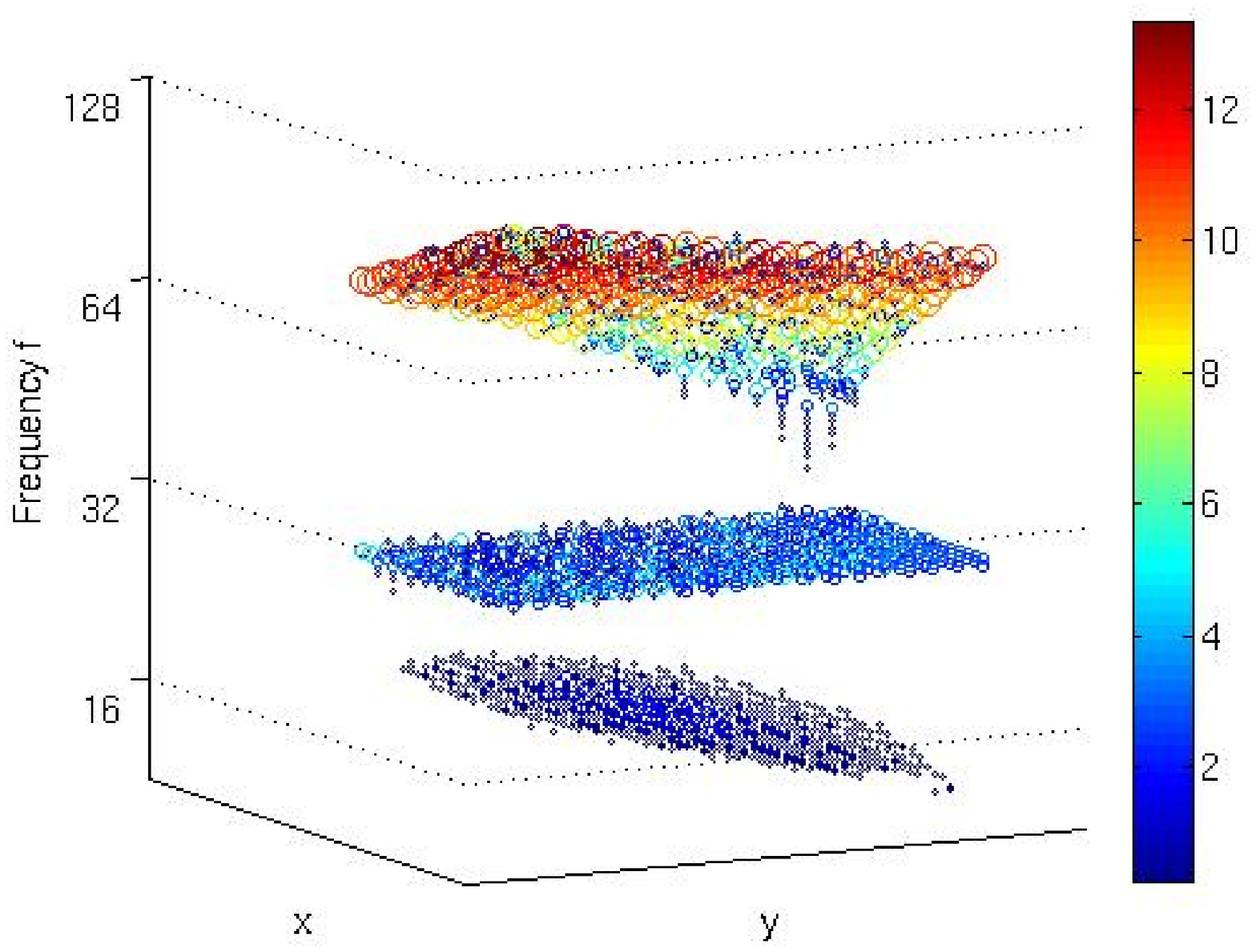}}
    \centerline{(g)}
  \end{minipage}
  \begin{minipage}[b]{0.24\linewidth}
    \centerline{
    \includegraphics[width=1.0\linewidth]{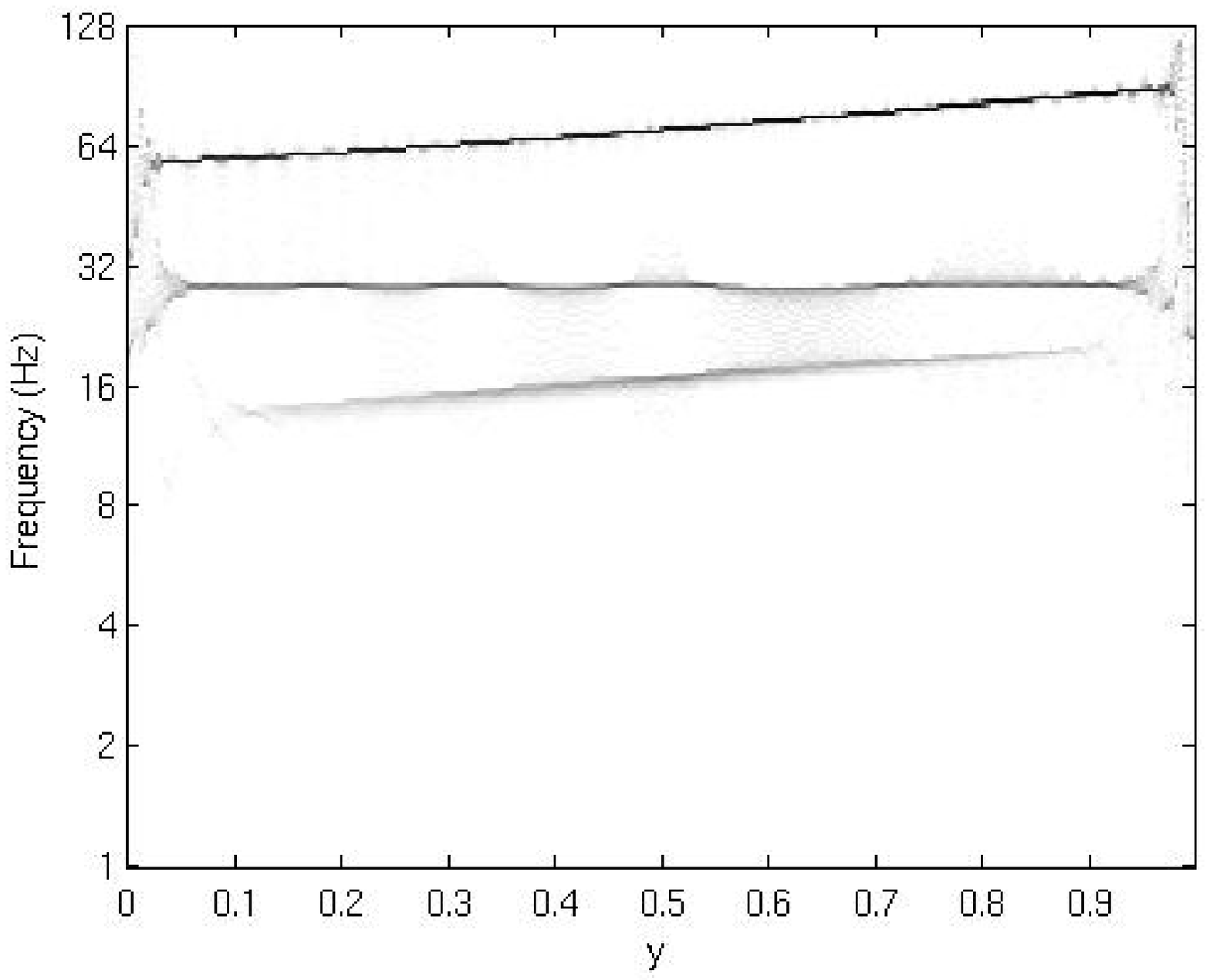}}
    \centerline{(h)}
  \end{minipage}
  \caption{(a): the 2D test-signal represented as an image. (b,c,d): its components. (e,f): the monogenic wavelet transform $|c_F|$ represented as a density in 3D, and in 2D for $y=0.5$ fixed. (g,h): the same visualizations of the MSST $|S_F|$.}
  \label{fig:visu}
\end{figure}
\subsection{Decomposition and Demodulation}
Once the MSST has been computed, the key point to separate and reconstruct the modes is to identify the instantaneous (anisotropic) frequency of each mode $|\nabla\varphi_\ell(b)|$ at each point $b$.  The $\ell^{th}$ IMMF $F_{\ell}$ is then approximated by $\widehat{F}_{\ell}(b)$ as follows:
\begin{equation}
\label{e:rec}
\widehat{F}_{\ell}(b) = \sum_{\hat k_{\ell}(b)-\kappa \le k \le \hat k_{\ell}(b)+\kappa } S_{F,\gamma}(b,k)\;.
\end{equation}
where $\hat k_{\ell}(b)$ is an appropriate estimate of $|\nabla\varphi_\ell(b)|$. Obtaining this estimate is an involved problem which will not be discussed here, so we simply use the approach used in ~\cite{daubechies:2011} based on a greedy algorithm. Our approximation $\widehat{F}_{\ell}(b)$ of the $\ell^{th}$ IMMF is computed by summing the coefficients of the synchrosqueezed transform in the vicinity of this instantaneous frequency, to regularize the solution. In practice the width $\kappa$
of the window is set to $5\Delta_k$.
To evaluate the method, we compare the mode $F_{\ell}$ of the test-signal defined in equation (\ref{e:testsig}) to the corresponding mode $\hat F_{\ell}$ obtained from the discrete MSST after extraction, by computing the normalized Mean-Squared Error (MSE):
\begin{equation}
\label{e:mse}
MSE(\hat F_{\ell}) = \frac{\lVert \hat F_{\ell} - F_{\ell} \rVert}{\lVert F_{\ell} \rVert}\;.
\end{equation}
Each reconstructed IMMF is displayed on Figure \ref{fig:extrac}, where one also provides the corresponding MSE.
It is clear that the modes are reconstructed with very high accuracy, although our method for ridge extraction is quite simple.

\begin{figure}[htb]
  \begin{minipage}[b]{0.32\linewidth}
    \centerline{
    \includegraphics[width=1.0\linewidth]{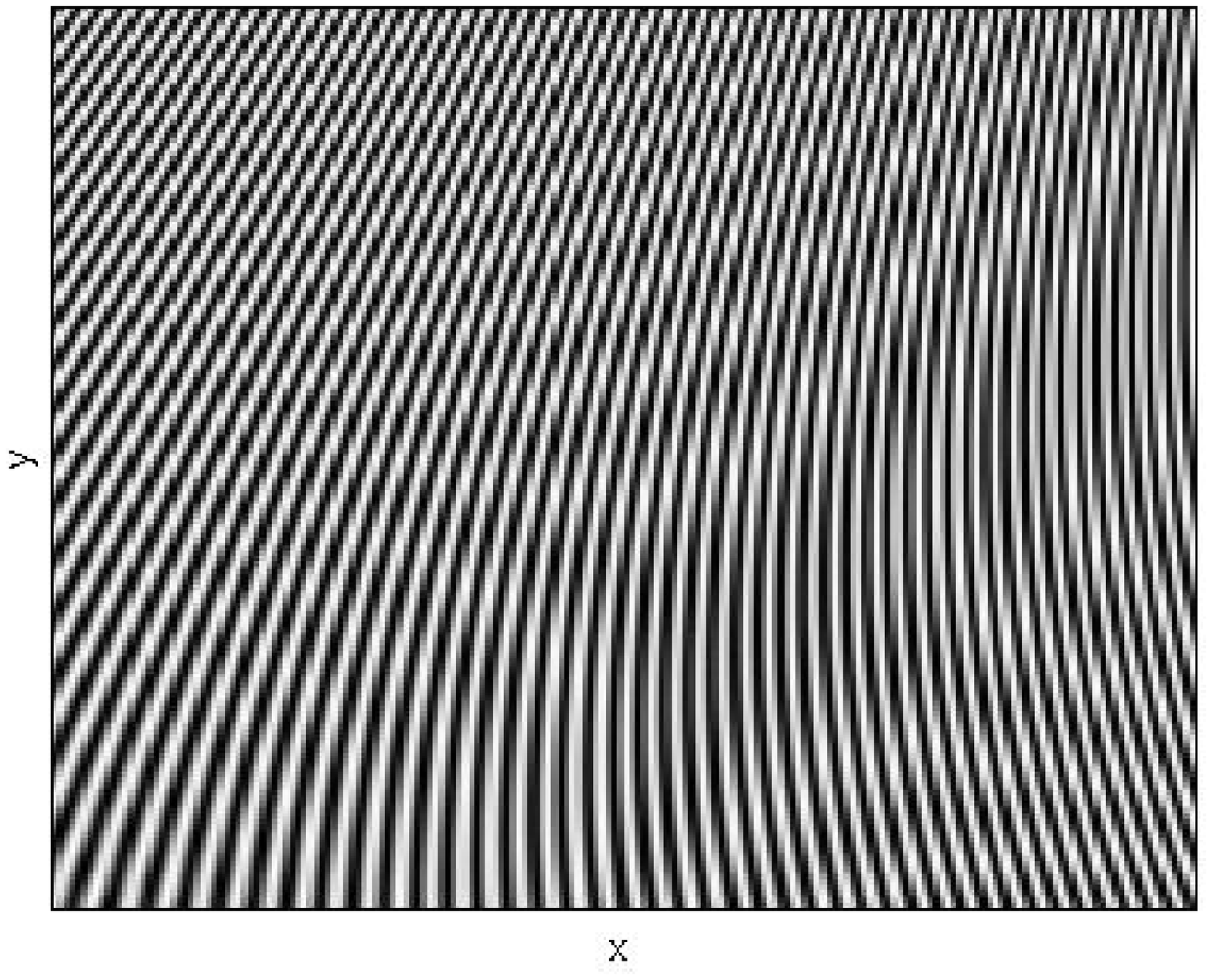}}
    \centerline{(a)}
  \end{minipage}
  \begin{minipage}[b]{0.32\linewidth}
    \centerline{
    \includegraphics[width=1.0\linewidth]{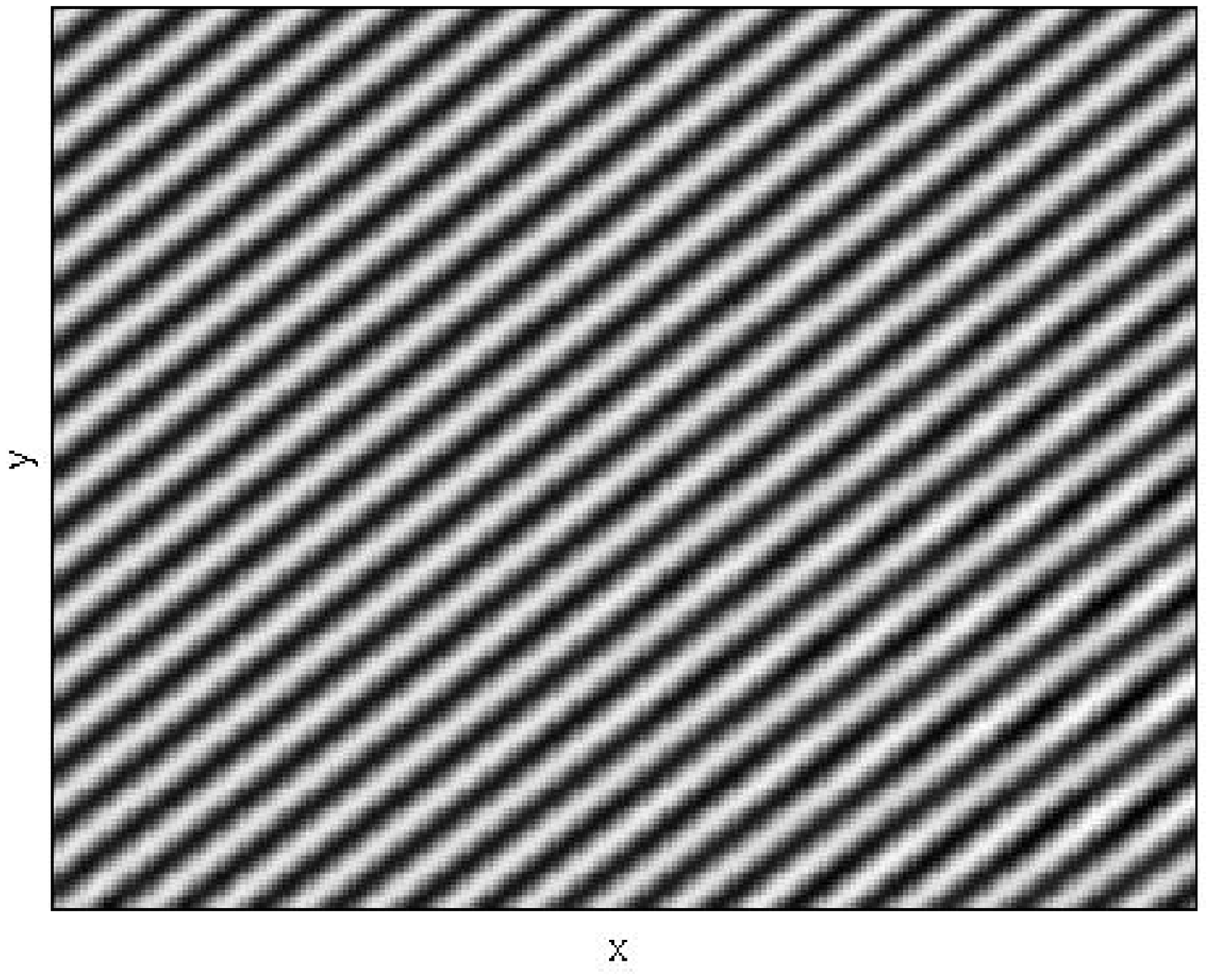}}
    \centerline{(b)}
  \end{minipage}
  \begin{minipage}[b]{0.32\linewidth}
    \centerline{
    \includegraphics[width=1.0\linewidth]{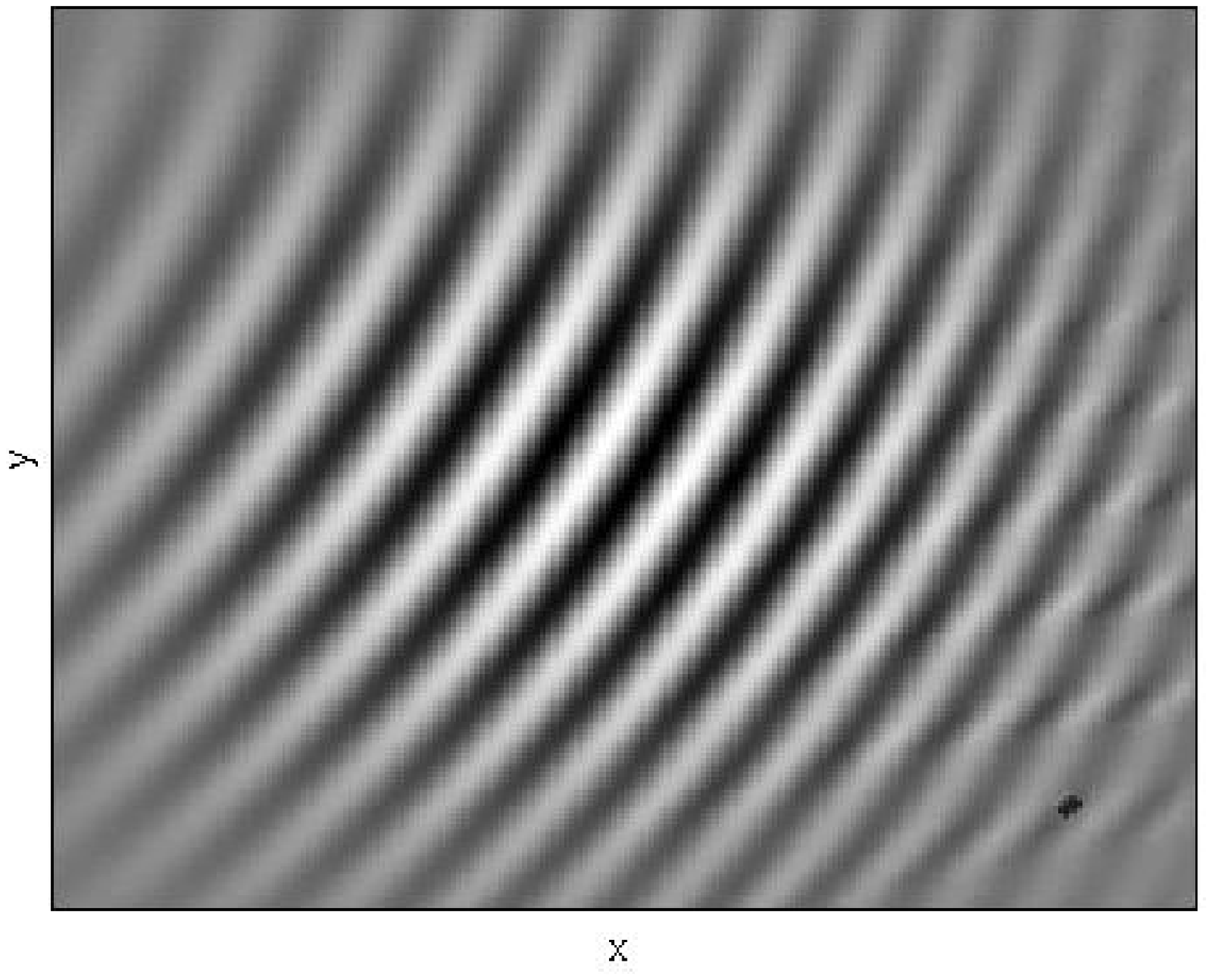}}
    \centerline{(c)}
  \end{minipage}
  \caption{The extracted and reconstructed modes 1, 2 and 3 for the test signal of equation (\ref{e:testsig}). The normalized MSE are respectively $0.03$, $0.06$ and $0.05$. We used the Morlet Wavelet $\psi_{\mu,\sigma}$ defined in~(\ref{e:morlet}) with $\sigma=2$ and $\mu=1$, and removed the border ($\frac{1}{8}$ at each side for each dimension) to avoid border effects.}
  \label{fig:extrac}
\end{figure}


Finally, we aim here at showing that the MSST remains efficient on real images or textures, even though they can not be considered as multicomponent signals. The following example shows the image Lenna where we artificially added an oscillating pattern. After a 2D synchrosqueezing transform one is able to extract the oscillating component to roughly reconstruct the original image. Note that this kind of image was first introduced in \cite{unser:vandeville:2009} to illustrate the discrete monogenic wavelet transform.

\begin{figure}[htb]
  \begin{minipage}[b]{0.32\linewidth}
    \centerline{
    \includegraphics[width=1.0\linewidth]{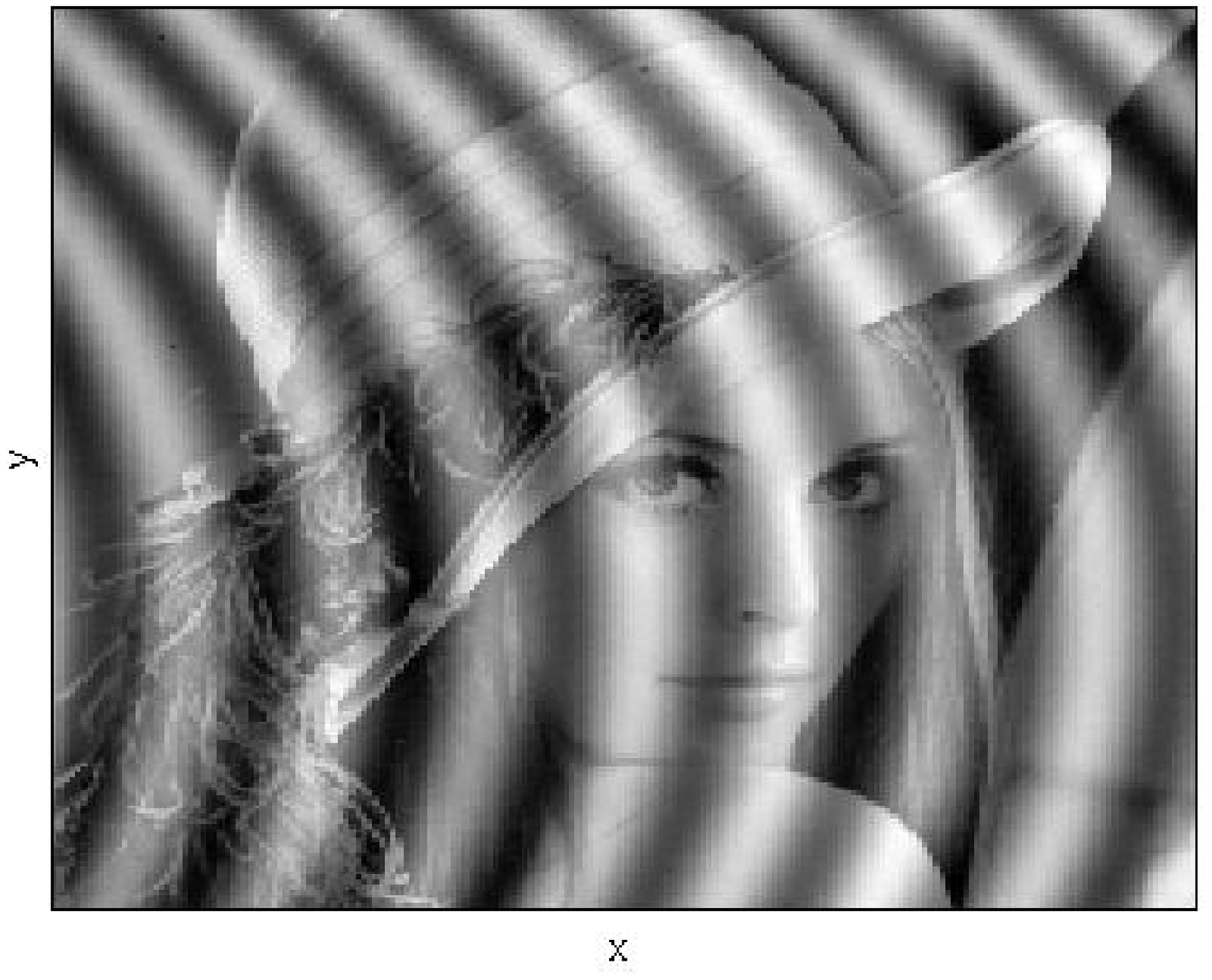}}
    \centerline{(a)}
  \end{minipage}
  \begin{minipage}[b]{0.32\linewidth}
    \centerline{
    \includegraphics[width=1.0\linewidth]{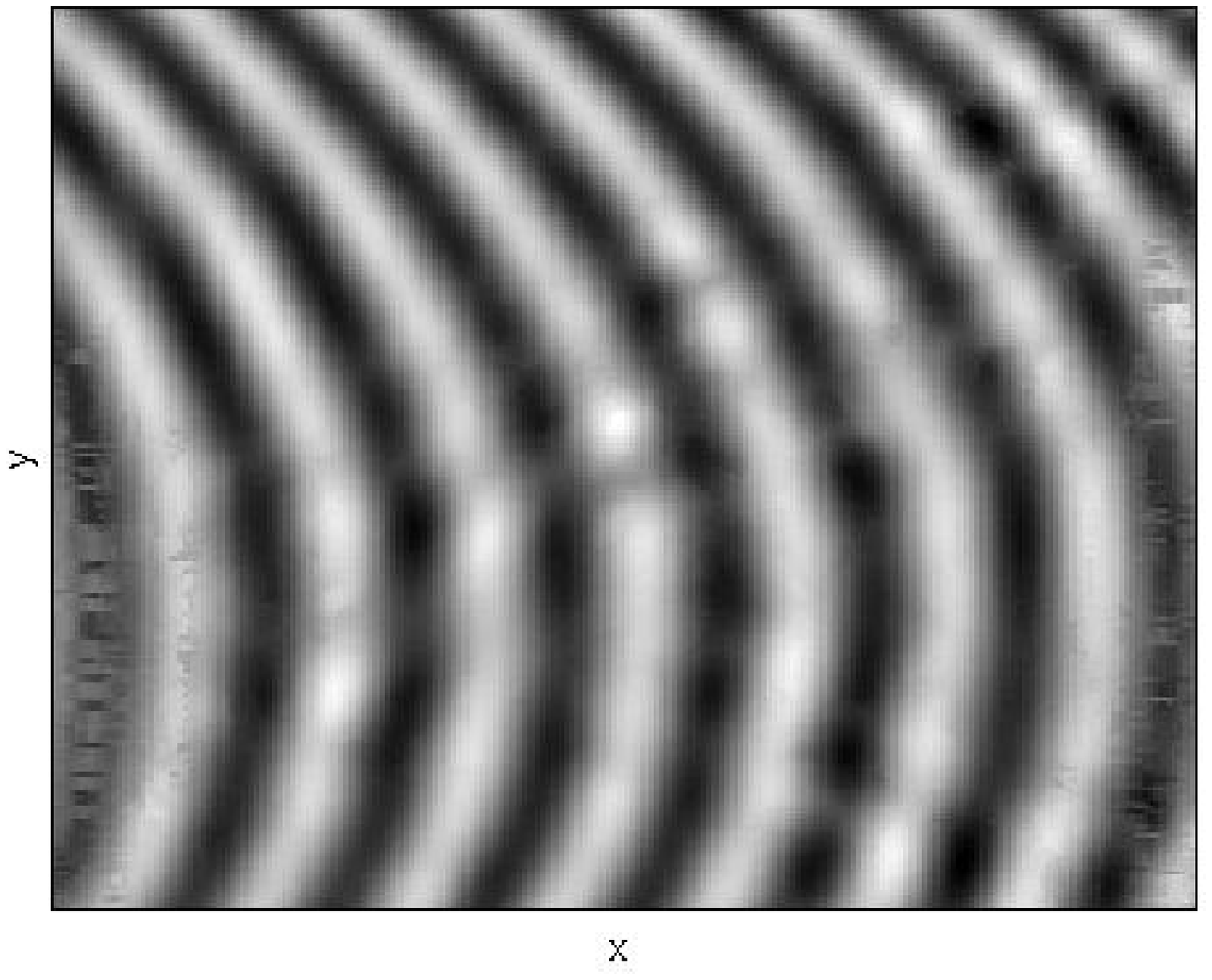}}
    \centerline{(b)}
  \end{minipage}
  \begin{minipage}[b]{0.32\linewidth}
    \centerline{
    \includegraphics[width=1.0\linewidth]{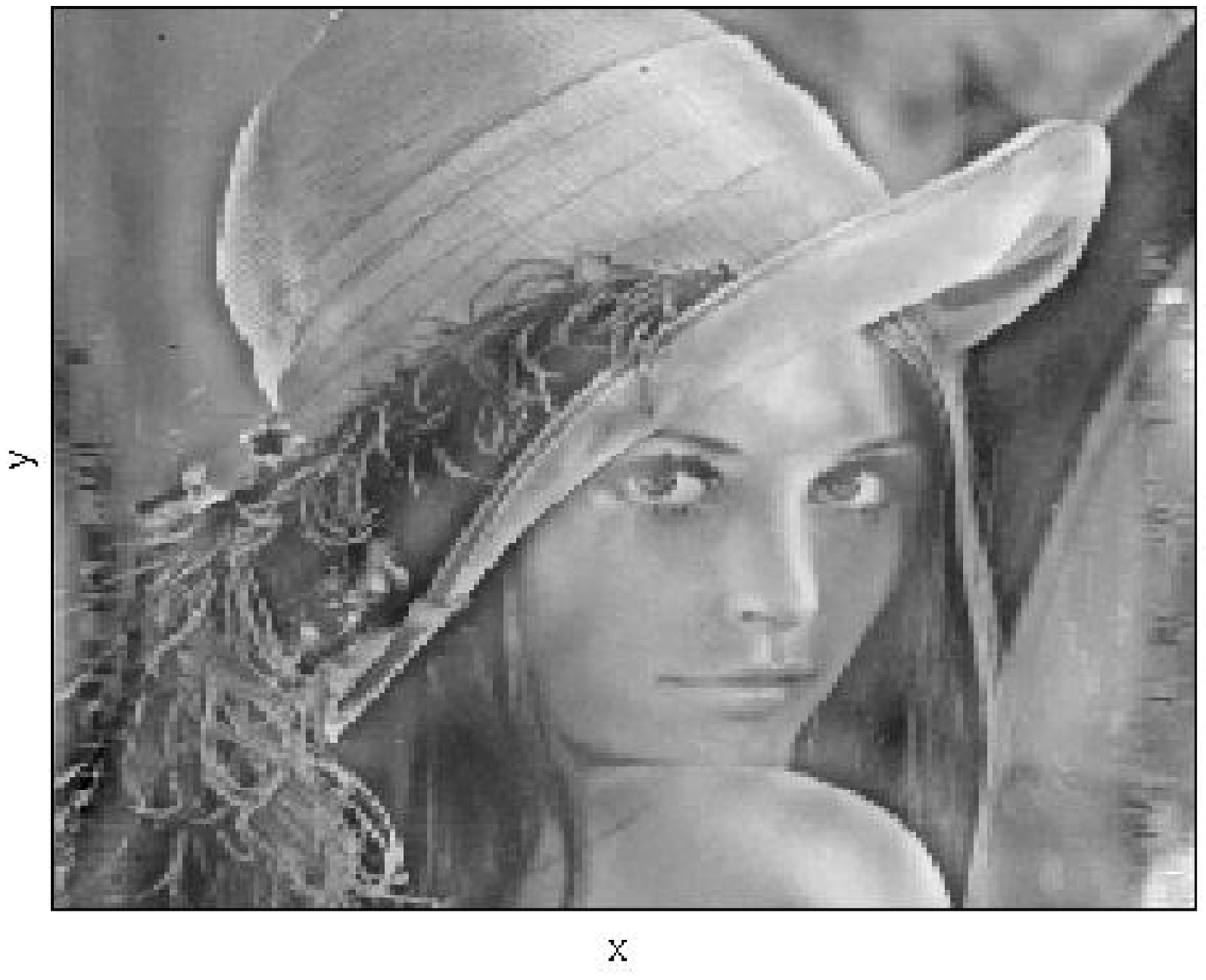}}
    \centerline{(c)}
  \end{minipage}
  \caption{Illustration of possible applications on real images. (a): Input image, Lenna plus an oscillating pattern. (b): the extracted mode by the 2D MSST. The parameters are the same as for Figure \ref{fig:extrac}. The MSE is 0.12. (c): The residual approximates the original Lenna image.}
  \label{fig:realim}
\end{figure}

%
\section{Proofs}\label{s:proofs}
\subsection{Proof of Proposition~\ref{pro:wavcoeffIMMF}}\label{s:proofs11}
We prove Proposition~\ref{pro:wavcoeffIMMF} into several steps.
\subsubsection{A preliminary result}
We first need a preliminary result:
\begin{lemma}\label{lem:interm1a}
Assume that $F$ is an IMMF with accuracy $\varepsilon$ of the form~(\ref{e:IMMF}). The following estimates hold for any $y,h\in\mathbb{R}^2$:
\begin{equation}\label{e:estimA}
|A(y+h)-A(y)|\leq \varepsilon \left(|h||\nabla\varphi(y)|+|h|^2 \frac{M}{2}\right)\;,
\end{equation}
\begin{equation}\label{e:estimT1}
|\cos\theta(y+h)-\cos\theta(y)|\leq \varepsilon \left(|h||\nabla\varphi(y)|+|h|^2 \frac{M}{2}\right)\;,
\end{equation}
\begin{equation}\label{e:estimT2}
|\sin\theta(y+h)-\sin\theta(y)|\leq \varepsilon \left(|h||\nabla\varphi(y)|+|h|^2 \frac{M}{2}\right)\;,
\end{equation}
\begin{equation}\label{e:estimP}
|\nabla \varphi(y+h)-\nabla\varphi(y)|\leq \varepsilon \left(|h||\nabla\varphi(y)|+|h|^2 \frac{M}{2}\right)\;,
\end{equation}
where $M$ is the constant defined in~(\ref{e:M}).
\end{lemma}
\begin{proof}
We only prove inequality~(\ref{e:estimA}), the others follows in the same way. Observe that:
\[
A(y+h)-A(y)=\int_0^1 \nabla A(y+th)\cdot h ~\rmd t
\]
By assumption~(\ref{e:hyp3a}), $\nabla A$ is slowly varying with respect to $\nabla \varphi$, then:
\[
|A(y+h)-A(y)|\leq \varepsilon|h|\int_0^1 |\nabla \varphi(y+th)|\rmd t
\]
Applying the usual mean value theorem, we deduce that for any $t$:
\begin{eqnarray*}
|\nabla \varphi(y+th)|&\leq& |\nabla \varphi(y)|+|\nabla \varphi(y+th)-\nabla \varphi(y)|\\
&\leq& |\nabla \varphi(y)|+t|h|\max_{i_1,i_2}\left(\sup_{y\in\mathbb{R}^2}|\partial^2_{x_{i_1},x_{i_2}} \varphi(y)|\right)\\
&\leq& |\nabla \varphi(y)|+M t|h|\;,
\end{eqnarray*}
applying Assumption~(\ref{e:M}) on second order partial derivatives of $\varphi$.
To sum up, one has
\[
|A(y+h)-A(y)|\leq \varepsilon|h|\int_0^1 |\nabla \varphi(y)|\rmd t+\varepsilon M|h|^2\int_0^1 t\rmd t\leq \varepsilon \left(|h||\nabla \varphi(y)|+M\frac{|h|^2}{2}\right)\;,
\]
which is the required result.
\end{proof}
%
\subsubsection{Proof of Proposition~\ref{pro:wavcoeffIMMF}}
We now prove Proposition~\ref{pro:wavcoeffIMMF}.\\
\noindent{\bf Proof of point~(\ref{pro:wavcoeffIMMFi}) of Proposition~\ref{pro:wavcoeffIMMF}}\\
By definition of the wavelet coefficients of $F$, one has
\[
c_F(a,b)=\int_{\mathbb{R}^2}A(x)\rme^{\varphi(x)n_{\theta(x)}}a^{-1}\psi\left(\frac{x-b}{a}\right)\rmd x\;.
\]
We now split this sum into three terms
\begin{equation}\label{e:split}
\begin{array}{lll}
c_F(a,b)&=&a^{-1} A(b)\int_{\mathbb{R}^2}\rme^{\varphi(x)n_{\theta(b)}}\psi\left(\frac{x-b}{a}\right)\rmd x \\
&&+a^{-1} A(b)\int_{\mathbb{R}^2}\left[\rme^{\varphi(x)n_{\theta(x)}}-\rme^{\varphi(x)n_{\theta(b)}}\right]
\psi\left(\frac{x-b}{a}\right)\rmd x\\
&&+a^{-1}\int_{\mathbb{R}^2}\left[A(x)-A(b)\right]\rme^{\varphi(x)n_{\theta(x)}}\psi\left(\frac{x-b}{a}\right)\rmd x\;.
\end{array}
\end{equation}
Observe now that
\begin{equation}\label{e:varphi}
\varphi(x)=\varphi(b)+\nabla\varphi(b)\cdot (x-b)\\
+\int_0^1 \left[\nabla \varphi (b+t(x-b))-\nabla\varphi (b)\right]\cdot(x-b)\rmd t\;.
\end{equation}
We set $u=\frac{x-b}{a}$ and use equation (\ref{ext}) with $k=\nabla \varphi(b)$. We then deduce that:
\begin{equation}\label{e:interm}
a^{-1} \int_{\mathbb{R}^2}\rme^{(\varphi(b)+\nabla\varphi(b)\cdot (x-b))n_{\theta(b)}}\psi\left(\frac{x-b}{a}\right)\rmd x
=\rme^{\varphi(b)n_{\theta(b)}}~a\widehat{\psi}(a\nabla \varphi(b))\;.
\end{equation}
Combining equations~(\ref{e:split}), (\ref{e:varphi}) and~(\ref{e:interm}) implies:
\begin{eqnarray*}\label{e:split2}
&&c_F(a,b)-a\widehat{\psi}(a\nabla \varphi(b)) \left(A(b)~\rme^{\varphi(b)n_{\theta(b)}}\right)\\
&=& A(b)\int_{\mathbb{R}^2}\rme^{(\varphi(b)+\nabla\varphi(b)\cdot (x-b))
n_{\theta(b)}}\left(\rme^{n_{\theta(b)}\int_0^1 \left[\nabla \varphi (b+t(x-b))-\nabla\varphi (b)\right]\cdot(x-b)\rmd t}-1\right)a^{-1}\psi\left(\frac{x-b}{a}\right)\rmd x \\
&&+ A(b)\int_{\mathbb{R}^2}\left[\rme^{\varphi(x)n_{\theta(x)}}-\rme^{\varphi(x)n_{\theta(b)}}\right]
a^{-1}\psi\left(\frac{x-b}{a}\right)\rmd x\\
&&+\int_{\mathbb{R}^2}\left[A(x)-A(b)\right]\rme^{\varphi(x)n_{\theta(x)}}a^{-1}\psi\left(\frac{x-b}{a}\right)\rmd x\;.
\end{eqnarray*}
Hence
\begin{eqnarray*}
&&\left|c_{F}(a,b)-a\widehat{\psi}(a\nabla \varphi(b))\left(A(b)~\rme^{\varphi(b)n_{\theta(b)}}\right)\right|\\
&\leq&a^{-1} A(b)\int_{\mathbb{R}^2}\left|\rme^{\left[\int_0^1 \left[\nabla \varphi
(b+t(x-b))-\nabla \varphi(b)\right]\cdot (x-b)\rmd t \right] n_{\theta(b)}}-1\right|\left|\psi\left(\frac{x-b}{a}\right)\right|\rmd x\\
&&+a^{-1}A(b)\int_{\mathbb{R}^2}\left|\rme^{\varphi(x)n_{\theta(x)}}-\rme^{\varphi(x)n_{\theta(b)}}\right|
\left|\psi\left(\frac{x-b}{a}\right)\right|\rmd x\\
&&+a^{-1}\int_{\mathbb{R}^2}\left|A(x)-A(b)\right|\left|\psi\left(\frac{x-b}{a}\right)\right|\rmd x\;.
\end{eqnarray*}
We now give an upper bound of each of these three terms. We bound the first one using the mean value theorem and the triangular inequality:
\begin{eqnarray*}
&&a^{-1} A(b)\int_{\mathbb{R}^2}\left|\rme^{\left[\int_0^1 \left(\left[\nabla \varphi (b+t(x-b))-\nabla \varphi(b)\right]\cdot (x-b)\right)\rmd t\right]n_{\theta(b)}}-1\right|\left|\psi\left(\frac{x-b}{a}\right)\right|\rmd x\\
&\leq&a^{-1}A(b)\int_{\mathbb{R}^2}\left[\int_0^1 \left|(\nabla \varphi (b+t(x-b))-\nabla \varphi(b))\cdot (x-b)\right|\rmd t\right]~\left|\psi\left(\frac{x-b}{a}\right)\right|\rmd x\;.
\end{eqnarray*}
We now use  inequality~(\ref{e:estimP}) with $y=b$ and $h=t(x-b)$. We then obtain:
\begin{eqnarray}
\nonumber
& &a^{-1}A(b)\int_{\mathbb{R}^2}\left[\int_0^1 \left|(\nabla \varphi (b+t(x-b))-\nabla \varphi(b))
\cdot (x-b)\right|\rmd t\right] \left|\psi\left(\frac{x-b}{a}\right)\right|\rmd x\\
\nonumber
&\leq&\varepsilon a^{-1} A(b)\int_{\mathbb{R}^2}\left[\int_0^1 (|t||x-b|^2\left|\nabla \varphi(b)\right|+ M|t|^2\frac{|x-b|^3}{2})\rmd t\right]\left|\psi\left(\frac{x-b}{a}\right)\right|\rmd x\\
\nonumber
&\leq&\varepsilon a^{-1}A(b)\int_{\mathbb{R}^2}\left[ \frac{|x-b|^2}{2}\left|\nabla \varphi(b)\right|+ M\frac{|x-b|^3}{6}\right]\cdot \left|\psi\left(\frac{x-b}{a}\right)\right|\rmd x\\
\label{term1}
&\leq&\varepsilon A(b)~\left(\frac{a^3}{2}\left|\nabla \varphi(b)\right|I_2+\frac{a^4}{6}M I_3\right)
\end{eqnarray}
where we set $u=\frac{x-b}{a}$ in the last integral. \\
\noindent Let us now give an upper bound of the second term. Since
\[
\rme^{\varphi(x)n_{\theta(x)}}-\rme^{\varphi(x)n_{\theta(b)}}=\sin(\varphi(x))(n_{\theta(x)}-n_{\theta(b)})\;,
\]
one has:
\begin{eqnarray}\nonumber
&&a^{-1}A(b)\int_{\mathbb{R}^2}\left|\rme^{\varphi(x)n_{\theta(x)}}-\rme^{\varphi(x)n_{\theta(b)}}\right|
\left|\psi\left(\frac{x-b}{a}\right)\right|\rmd x\;\\
\nonumber
&&\leq a^{-1} A(b)\int_{\mathbb{R}^2}\left|n_{\theta(x)}-n_{\theta(b)}\right|\left|
\psi\left(\frac{x-b}{a}\right)\right|\rmd x\\
\nonumber
&&\leq \varepsilon a^{-1} A(b)\sqrt{2}\int_{\mathbb{R}^2}\left(|x-b||\nabla\varphi(b)|+|x-b|^2 \frac{M}{2}\right)
\left|\psi\left(\frac{x-b}{a}\right)\right|\rmd x\\
\nonumber
&&\leq \varepsilon a^2 A(b)\sqrt{2}\int_{\mathbb{R}^2}\left(|u||\nabla\varphi(b)|+ a |u|^2 \frac{M}{2}\right)
|\psi(u)| \rmd u\\
\label{term2}
&&\leq \varepsilon A(b)\sqrt{2}~\left(a^2|\nabla\varphi(b)|I_1+ a^3 \frac{M}{2}I_2\right)
\end{eqnarray}
where we have used inequalities~(\ref{e:estimT1}) and~(\ref{e:estimT2}) with $y=b$ and $h=x-b$ for the third inequality and set $u=(x-b)/a$ for the fourth one.\\
\noindent Finally, using the same approach we prove that the last term is bounded by, using~(\ref{e:estimA}):
\begin{equation}
\label{term3}
a^{-1}\int_{\mathbb{R}^2}\left|A(x)-A(b)\right|\left|\psi\left(\frac{x-b}{a}\right) \right|\rmd x
 \leq \varepsilon ~\left(a^2|\nabla\varphi(b)|I_1+ a^3 \frac{M}{2}I_2\right)\;.
 \end{equation}
Combining inequalities~(\ref{term1}, \ref{term2}, \ref{term3}) leads to:
\begin{eqnarray*}
\left|c_{F}(a,b)-a\widehat{\psi}(a\nabla \varphi(b))A(b)\rme^{\varphi(b)n_{\theta(b)}}\right|
&\leq&\varepsilon a^2A(b)\left( \left(\frac{a}{2}\left|\nabla \varphi(b)\right|I_2+\frac{a^2}{6}M I_3\right)
+\sqrt{2}\left(|\nabla\varphi(b)|I_1+ a \frac{M}{2}I_2\right)\right)\\
&+& \varepsilon a^2~\left(|\nabla\varphi(b)|I_1+ a \frac{M}{2}I_2\right) \\
\end{eqnarray*}
which gives (\ref{e:remaind1}).

\

%
%
\noindent{\bf Proof of point~(\ref{pro:wavcoeffIMMFii}) of Proposition~\ref{pro:wavcoeffIMMF}}\\
Since $A\in L^\infty(\mathbb{R}^2)$ and $\psi\in W^{1,1}(\mathbb{R}^2)$, one has:
\[
\partial_{b_i} c_F(a,b)=- \int_{\mathbb{R}^2}A(x)\rme^{\varphi (x)n_{\theta}(x)}a^{-2}\partial_{x_i}\psi\left(\frac{x-b}{a}\right)\rmd x\;.
\]
As below we split the sum into three terms
\begin{equation}\label{e:split:der}
\begin{array}{lll}
\partial_{b_i} c_F(a,b)&=&-a^{-2} A(b)\int_{\mathbb{R}^2}\rme^{\varphi (x)n_{\theta}(b)}\partial_{x_i}\psi\left(\frac{x-b}{a}\right)\rmd x\\
&&-a^{-2} A(b)\int_{\mathbb{R}^2}\left(\rme^{\varphi (x)n_{\theta}(x)}-\rme^{\varphi (x)n_{\theta}(b)}\right)\partial_{x_i}\psi\left(\frac{x-b}{a}\right)\rmd x\\
&&-a^{-2}\int_{\mathbb{R}^2}\left[A(x)-A(b)\right]\rme^{\varphi (x)n_{\theta}(x)}\partial_{x_i}\psi\left(\frac{x-b}{a}\right)\rmd x
\end{array}
\end{equation}
We again use~(\ref{e:varphi}) and equation (\ref{ext}) with $k=\nabla \varphi(b)$ to write:
\begin{equation}\label{e:interm2}
-a^{-2} \int_{\mathbb{R}^2}\rme^{(\varphi(b)+\nabla\varphi(b)\cdot (x-b))n_{\theta(b)}}\partial_{x_i}\psi\left(\frac{x-b}{a}\right)\rmd x \\
= \rme^{\varphi(b)n_{\theta(b)}}~\partial_{b_i}\varphi(b)n_{\theta(b)} ~a \widehat{\psi}(a\nabla \varphi(b))\;.
\end{equation}
Hence we deduce that
\begin{eqnarray*}
&&\left|\partial_{b_i} c_F(a,b)-\partial_{b_i} \varphi(b) n_{\theta(b)}\left(a\widehat{\psi}(a\nabla \varphi(b))\right)\left(A(b)\rme^{\varphi(b)n_{\theta(b)}}\right)\right|\\
&\leq&a^{-2} A(b)\int_{\mathbb{R}^2}\left|\rme^{(\int_0^1 \left[\nabla \varphi
(b+t(x-b))-\nabla \varphi(b)\right]\cdot (x-b)\rmd t )n_{\theta(b)}}-1\right|\left|\partial_{x_i}\psi\left(\frac{x-b}{a}\right)\right|\rmd x\\
&&+a^{-2} A(b)\sqrt{2}\int_{\mathbb{R}^2}\left|n_{\theta(x)}-n_{\theta(b)}\right|\left|\partial_{x_i}\psi\left(\frac{x-b}{a}\right)\right|\rmd x\\
&&a^{-2} \int_{\mathbb{R}^2}\left|A(x)-A(b)\right|\left|\partial_{x_i}\psi\left(\frac{x-b}{a}\right)\right|\rmd x
\end{eqnarray*}
A similar approach than in the proof of point~(\ref{pro:wavcoeffIMMFi}) leads to the following upper bound:
\begin{eqnarray*}
&&\left|\partial_{b_i} c_F(a,b)-\partial_{b_i} \varphi(b) n_{\theta(b)}\left(a\widehat{\psi}(a\nabla \varphi(b))\right)\left(A(b)\rme^{\varphi(b)n_{\theta(b)}}\right)\right|\\
&\leq& \varepsilon A(b)~\left(\frac{a^2}{2}\left|\nabla \varphi(b)\right|I'_2+\frac{a^3}{6}M I'_3\right)\\
&&+\varepsilon A(b) \sqrt{2}~\left(a|\nabla\varphi(b)|I'_1+ a^2 \frac{M}{2}I'_2\right)\\
&&+\varepsilon ~\left(a|\nabla\varphi(b)|I'_1+ a^2 \frac{M}{2}I'_2\right)
\end{eqnarray*}
which leads to the estimate (\ref{pro:wavcoeffIMMFii}) of Proposition~\ref{pro:wavcoeffIMMF}, with:
$$
|R_2(a,b)|\leq A(b)\left(\frac{a}{2}\left|\nabla \varphi(b)\right|I'_2+\frac{a^2}{6}M I'_3\right)+ \left(\sqrt{2}A(b)+1\right)\left(|\nabla\varphi(b)|I'_1+ a \frac{M}{2}I'_2\right)\;.
$$

%
%
\subsection{Proof of Theorem~\ref{th:main1}}\label{s:proofs12}
We can now prove Theorem~\ref{th:main1}. By definition for $i=1,2$:
\[
\Lambda_i(a,b)=\left(\partial_{b_i} c_{F}(a,b)\right)(c_F(a,b))^{-1}\;.
\]
Set $B=a\widehat{\psi}(a\nabla\varphi(b))\left(A(b)\rme^{\varphi(b)n_{\theta(b)}}\right)$. Proposition~\ref{pro:wavcoeffIMMF} implies that:
\begin{equation}\label{e:3}
|c_F(a,b)-B|\leq \varepsilon a^2 \left|R_1(a,b)\right|\;,
\end{equation}
and
\begin{equation}\label{e:4}
|\partial_{b_i}c_F(a,b)-\partial_{b_i}\varphi(b) n_\theta(b)B|\leq \varepsilon a \left| R_2(a,b)\right|\;,
\end{equation}
where $R_1(a,b)$ and $R_2(a,b)$ satisfy inequalities (\ref{e:remaind1}) and (\ref{e:remaind2}) respectively.
Then, for $i=1,2$:
\[
\Lambda_i(a,b)-\partial_{b_i}\varphi(b)n_{\theta(b)}=\left[\partial_{b_i}c_F(a,b)-\partial_{b_i}\varphi(b) n_{\theta(b)} B+\partial_{b_i}\varphi(b) n_{\theta(b)}(B-c_F(a,b))\right]\left[c_F(a,b)\right]^{-1}\;.
\]
Using~(\ref{e:3}) and~(\ref{e:4}) implies that:
\[
|\Lambda_i(a,b)-\partial_{b_i}\varphi(b)n_{\theta(b)}|\leq \varepsilon\left[a|R_2(a,b)|+ a^2|R_1(a,b)||\partial_{b_i}\varphi(b)|\right]~|c_F(a,b)|^{-1}\;.
\]
By assumption $|c_F(a,b)|\geq \varepsilon^\nu$ for all $(a,b)$, which leads to (\ref{e:thmain1}). The last estimate (\ref{e:thmain1bis}) follows from the fact that for any given $a_0>0$ \[
\sup_{(a,b)\in (0,a_0)\times\mathbb{R}} \left(a|R_2(a,b)|+ a^2|R_1(a,b)||\nabla\varphi(b)|\right)<\infty\;.
\]
Then, since $\nu\in (0,1/2)$, there exists some $\varepsilon_0>0$ depending on $a_0$ such that for any $0<\varepsilon\leq \varepsilon_0$
$$
a|R_2(a,b)|+ a^2|R_1(a,b)||\nabla\varphi(b)| \leq \varepsilon^{2\nu-1}\;,
$$
which is equivalent to
$$
\varepsilon^{1-\nu}\left(a|R_2(a,b)|+ a^2|R_1(a,b)||\nabla\varphi(b)|\right) \leq \varepsilon^{\nu}\;.
$$
This last inequality and inequality~(\ref{e:thmain1}) clearly implies inequality~(\ref{e:thmain1bis}).
\subsection{Proof of Proposition~\ref{pro:main1}}\label{s:proofs2}
The proof of Proposition~\ref{pro:main1} relies on Theorem~\ref{th:main1} and on the following lemma:
\begin{lemma}\label{lem:promain1}
Let $F$ an IMMF with accuracy $\varepsilon>0$. Provided that $\varepsilon>0$ is sufficiently small, the sign of $\mathrm{Re}(\partial_{b_1} c_F(a,b)\overline{\partial_{b_2} c_F(a,b)})$ is this of $\partial_{b_1} \varphi(b)\partial_{b_2} \varphi(b)$.
\end{lemma}
\begin{proof}
Let us first observe that Theorem~\ref{th:main1} implies that for all $(a,b)$ under consideration:
\begin{eqnarray*}
\Lambda_1(a,b)&=&\partial_{b_1}\varphi(b)n_{\theta(b)}+O( \varepsilon^{\nu})\\
\Lambda_2(a,b)&=&\partial_{b_2}\varphi(b)n_{\theta(b)}+O(\varepsilon^{\nu})\
\end{eqnarray*}
Then
\begin{equation}\label{e:l}
\Lambda_1(a,b)\overline{\Lambda_2(a,b)} =\partial_{b_1}\varphi(b) \partial_{b_2}\varphi(b)+O(\varepsilon^{\nu})\;.
\end{equation}
Hence, for $\varepsilon>0$ sufficiently small, the sign of $\Lambda_1(a,b)\overline{\Lambda_2(a,b)}$ is this of $\partial_{b_1} \varphi(b)\partial_{b_2} \varphi(b)$.

To get the required conclusion, we now relate $\Lambda_1(a,b)\overline{\Lambda_2(a,b)}$ and $\partial_{b_1} c_{F}(a,b)\overline{\partial_{b_2} c_{F}(a,b)}$. Let us first remark that the definition of the two Clifford vectors $\Lambda_1(a,b),\Lambda_2(a,b)$ from equation~(\ref{e:L1}) implies that :
\begin{equation}\label{e:Lambda:def}
{\rm Re}\left(\Lambda_1(a,b)\overline{\Lambda_2(a,b)}\right) ={\rm Re} \left( \partial_{b_1} c_{F}(a,b)~(c_{F}(a,b))^{-1}\overline{\partial_{b_2} c_{F}(a,b)~(c_{F}(a,b))^{-1}} \right)\;.
\end{equation}
We use now the fact that $\partial_{b_i} c_{F}(a,b)$ and $c_F(a,b)$ are both Clifford vectors. We then apply equality~(\ref{e:multclifford}) with $q=\partial_{b_2} c_{F}(a,b)$ and $q'=(c_{F}(a,b))^{-1}$. Hence we get that
\[
\overline{\partial_{b_2} c_{F}(a,b)~(c_{F}(a,b))^{-1}} =\overline{(c_{F}(a,b))^{-1}}\times\overline{\partial_{b_2} c_{F}(a,b)}\;.
\]
Combining this last equality with~(\ref{e:Lambda:def}) leads to
\begin{equation}\label{e:Lambdas}
{\rm Re}\left(\Lambda_1(a,b)\overline{\Lambda_2(a,b)}\right)=  {\rm Re} \left(\partial_{b_1} c_{F}(a,b)~\overline{\partial_{b_2} c_{F}(a,b)}\right)|c_{F}(a,b)|^{-2}\;.
\end{equation}
Taking into account~(\ref{e:Lambdas}), equation~\ref{e:l} reads
$$
 \frac{ {\rm Re} \left(\partial_{b_1} c_{F}(a,b)~\overline{\partial_{b_2} c_{F}(a,b)}\right) }{|c_{F}(a,b)|^2}
 =\partial_{b_1}\varphi(b) \partial_{b_2}\varphi(b)+O(\varepsilon^{\nu})\;,
$$
which implies that ${\rm Re} \left(\partial_{b_1} c_{F}(a,b)~\overline{\partial_{b_2} c_{F}(a,b)}\right)$ and $\partial_{b_1}\varphi(b) \partial_{b_2}\varphi(b)$ have the same sign for $\varepsilon>0$ sufficiently small.
\end{proof}
%
%
\noindent We can now prove Proposition~\ref{pro:main1}.\\
\noindent{\bf Proof of Proposition~\ref{pro:main1}}\\
Theorem~\ref{th:main1} directly implies that for any $i=1,2$ and $(a,b)$ under consideration:
\[
||\Lambda_i(a,b)|-|\partial_{b_i}\varphi(b)||\leq \varepsilon^{\nu}\;.
\]
Since $\partial_{b_1}\varphi(b)$ is assumed to be always positive, we get~(\ref{e:promain1b}). Since $\partial_{b_2}\varphi(b)$ and $\partial_{b_1}\varphi(b)\partial_{b_2}\varphi(b)$ have the same sign, Lemma~\ref{lem:promain1} implies~(\ref{e:promain1c}).
\subsection{Proof of Theorem~\ref{th:main2}}\label{s:proofs:th:main2}
The proof of Theorem~\ref{th:main2} consists in several steps.
\begin{lemma}\label{lem:interm3}
Let $F$ be a function belonging to $\mathcal{A}_{\varepsilon,d}$. For any $(a,b)$, there can be at most one $\ell\in\{1,\cdots,L\}$ such that
\begin{equation}\label{e:sep2}
\left|a|\nabla \varphi_{\ell}(b)|-1\right|<\Delta\;.
\end{equation}
\end{lemma}
\begin{remark}\label{rem:sep}
Condition~(\ref{e:sep2}) is a necessary condition for having $\widehat{\psi}(a\nabla\varphi_\ell(b))\neq 0$. Lemma~\ref{lem:interm3} then states there is at most one $\ell$ such that $\widehat{\psi}(a\nabla\varphi_\ell(b))\neq 0$. Then the two following sums, involved in the estimation of $c_F(a,b)$ and $\partial_{b_i}c_F(a,b)$:
\[
\sum_{\ell=1}^L A_\ell(b)~\rme^{\varphi_\ell(b)n_{\theta_\ell(b)}}~a\widehat{\psi}(a\nabla \varphi_\ell(b))\;,
\]
and
\[
\sum_{\ell=1}^L A_\ell(b)~\rme^{\varphi_\ell(b)n_{\theta_\ell(b)}}~a\widehat{\psi}(a\nabla \varphi_\ell(b))\partial_i \varphi_\ell(b)n_{\theta_\ell(b)}\;,
\]
contain at most one term.
\end{remark}
\begin{proof}
We follow the same line as in~\cite{daubechies:2011}. Assume that there exists some $(\ell_1, \ell_2)$ such that~(\ref{e:sep2}) is satisfied with $\ell_1\neq \ell_2$. One may suppose that $\ell_1< \ell_2$. One then has for $j=1,2$,
\[
a^{-2}(1-\Delta)^2<\left|\frac{\partial\varphi_{\ell_j}(b)}{\partial x_1}\right|^2+\left|\frac{\partial\varphi_{\ell_j}(b)}{\partial x_2}\right|^2<a^{-2}(1+\Delta)^2\;,
\]
which directly implies
\[
\left|\frac{\partial\varphi_{\ell_1}(b)}{\partial x_1}\right|^2+\left|\frac{\partial\varphi_{\ell_1}(b)}{\partial x_2}\right|^2+\left|\frac{\partial\varphi_{\ell_2}(b)}{\partial x_1}\right|^2+\left|\frac{\partial\varphi_{\ell_2}(b)}{\partial x_2}\right|^2> 2a^{-2}(1-\Delta)^2\;,
\]
and
\[
\left|\frac{\partial\varphi_{\ell_2}(b)}{\partial x_1}\right|^2+\left|\frac{\partial\varphi_{\ell_2}(b)}{\partial x_2}\right|^2-
\left|\frac{\partial\varphi_{\ell_1}(b)}{\partial x_1}\right|^2-\left|\frac{\partial\varphi_{\ell_1}(b)}{\partial x_2}\right|^2< 4 a^{-2}\Delta\;.
\]
Further, by assumption~(\ref{e:sep}), for any $i=1,2$ one has
\[
\left|\frac{\partial\varphi_{\ell_2}(b)}{\partial x_i}\right|^2-\left|\frac{\partial\varphi_{\ell_1}(b)}{\partial x_i}\right|^2\geq\left|\frac{\partial\varphi_{\ell_2}(b)}{\partial x_i}\right|^2-\left|\frac{\partial\varphi_{\ell_2-1}(b)}{\partial x_i}\right|^2
\geq d\left[\left|\frac{\partial\varphi_{\ell_2}(b)}{\partial x_i}\right|+\left|\frac{\partial\varphi_{\ell_2-1}(b)}{\partial x_i}\right|\right]^2\;.
\]
where we applied assumption (\ref{e:sepbis}):
\begin{equation*}
\left||\partial_{x_i}\varphi_{\ell_2}(x)|-|\partial_{x_i}\varphi_{\ell_2-1}(x)|\right|
\geq \left|\partial_{x_i}\varphi_{\ell_2}(x)n_{\theta_\ell(x)}-\partial_{x_i}\varphi_{\ell_2-1}(x)n_{\theta_{\ell'}(x)}\right|
\geq d~\left[ \left|\partial_{x_i}\varphi_{\ell_2}(x)\right|+\left|\partial_{x_i}\varphi_{\ell_2-1}(x)\right|\right]\;.
\end{equation*}
The usual inequality $(u+v)^2\geq u^2+v^2$, valid for any $u,v\geq 0$ implies that
\[
\left|\frac{\partial\varphi_{\ell_2}(b)}{\partial x_i}\right|^2-\left|\frac{\partial\varphi_{\ell_1}(b)}{\partial x_i}\right|^2
\geq d\left[\left|\frac{\partial\varphi_{\ell_2}(b)}{\partial x_i}\right|^2+\left|\frac{\partial\varphi_{\ell_1}(b)}{\partial x_i}\right|^2\right]\;.
\]
Summing over $i=1,2$ gets:
\[
4 a^{-2}\Delta> 2d a^{-2}(1-\Delta)^2\geq 2 d a^{-2}(1-2\Delta)\;,
\]
that is
\[
\Delta > d/\left[2(1+d)\right]\;.
\]
which leads to a contradiction.
\end{proof}
\begin{lemma}\label{lem:interm2b}
Let $F$ be a superposition of IMMF with accuracy $\varepsilon$ and separation $d$ of the form~(\ref{e:supIMMF}). For all $(a,b)\in\mathbb{R}^*_+\times\mathbb{R}^2$
one has
\[
\left|c_{F}(a,b)-a\sum_{\ell'=1}^L \widehat{\psi}(a\nabla \varphi_{\ell}(b))\left(A_{\ell}(b)~\rme^{\varphi_{\ell}(b)n_{\theta_{\ell}(b)}}\right)\right|\leq
 \varepsilon a^2 \Gamma_1(a,b)\;.
\]
with
\begin{equation}\label{e:Gamma1}\begin{array}{lll}
 \Gamma_1(a,b) &=&I_1 \sum_{\ell=1}^L (\sqrt{2}A_\ell(b)+1)~|\nabla\varphi_\ell(b)|+a ~\frac{I_2}{2}\left[ \sum_{\ell=1}^L A_\ell(b)~\left(\left|\nabla \varphi_\ell(b)\right| +  M_{\ell}\right)\right] \\
&& +a ~\frac{I_2}{2}\sum_{\ell=1}^L M_{\ell}
 +a^{2}~\frac{I_3}{6}\sum_{\ell=1}^L A_\ell(b)M_{\ell}\;.
\end{array}\end{equation}
($I_\alpha$ being introduced in (\ref{integrales})).
In particular
\begin{itemize}
\item If for some $\ell_0\in\{1,\cdots,L\}$
\begin{equation}\label{e:sep3}
\left|a|\nabla\varphi_{\ell_0}(b)|-1\right|<\Delta\;.
\end{equation}
then
\begin{equation}\label{e:lem4a}
\left|c_{F}(a,b)-a\widehat{\psi}(a\nabla \varphi_{\ell_0}(b))A_{\ell_0}(b)~\rme^{\varphi_{\ell_0}(b)n_{\theta_{\ell_0}(b)}}\right|\leq
 \varepsilon a^2 \Gamma_1(a,b)\;.
\end{equation}
\item If for all $\ell\in\{1,\cdots,L\}$, Condition~(\ref{e:sep3}) is not satisfied then
\begin{equation}\label{e:lem4b}
\left|c_{F}(a,b)\right|\leq
 \varepsilon a^2 \Gamma_1(a,b)\;.
\end{equation}
\end{itemize}
\end{lemma}
\begin{proof}
By assumption $F$ is of the form
\[
F=\sum_{\ell=1}^L A_{\ell}\rme^{\varphi_\ell n_\ell}\;,
\]
where for each $\ell$, $F_\ell=A_{\ell}\rme^{\varphi_\ell n_\ell}$ is an IMMF of accuracy $\varepsilon$. Proposition \ref{pro:wavcoeffIMMF} applied successively to $F_1,\cdots,F_L$ and the linearity of the continuous wavelet transform imply:
\[
\left|c_{F}(a,b)-a\sum_{\ell=1}^L \widehat{\psi}(a\nabla \varphi_{\ell}(b))A_{\ell}(b)~\rme^{\varphi_{\ell}(b)n_{\theta_{\ell}(b)}}\right|\leq
 \varepsilon a^2 \Gamma_1(a,b)\;.
\]
The rest of the proof follows from Remark~\ref{rem:sep}, which states that under condition~(\ref{e:sep3}), the following sum
\[
\sum_{\ell=1}^L A_{\ell}(b)~\rme^{\varphi_{\ell}(b)n_{\theta_{\ell}(b)}}~a\widehat{\psi}(a\nabla \varphi_{\ell}(b))\;,
\]
reduces to $A_{\ell_0}(b)~\rme^{\varphi_{\ell_0}(b)n_{\theta_{\ell_0}(b)}}~a\widehat{\psi}(a\nabla \varphi_{\ell_0}(b))$, and 0 if condition~(\ref{e:sep3}) is never satisfied.
\end{proof}
\begin{lemma}\label{lem:interm4b}
Let $F$ be a superposition of IMMF with accuracy $\varepsilon$ and separation $d$. For all $a,b$ such that
\begin{equation}\label{e:sep4}
\left|a|\nabla\varphi_\ell(b)|-1\right|<\Delta\;,
\end{equation}
one has, for any $i=1,2$:
\begin{equation}\label{e:lem4b2}
\left|\partial_{b_i} c_F(a,b)-a \partial_i \varphi_\ell(b)n_{\theta_\ell(b)}\widehat{\psi}(a\nabla \varphi_\ell(b))\left(A_\ell(b)\rme^{\varphi_\ell(b)n_{\theta_\ell(b)}}\right)\right|
\leq \varepsilon a \Gamma_2(a,b)\;,
\end{equation}
with
\begin{equation}\label{e:Gamma2}\begin{array}{lll}
\Gamma_2(a,b)&=&I'_1 \sum_{\ell=1}^L |\nabla\varphi_\ell(b)| ( \sqrt{2}A_\ell(b) +1)
+a\frac{I'_2}{2}\sum_{\ell=1}^L \left( M_{\ell}+A_\ell(b) (\sqrt{2}M_{\ell}+|\nabla\varphi_\ell(b)|)\right)\\
&&+~ a^2 \frac{I'_3}{6} \sum_{\ell=1}^L A_\ell(b)~M_{\ell}
\end{array}\end{equation}
($I'_\alpha$ being introduced in (\ref{integrales})).
\end{lemma}
\begin{proof}
As in the proof of Lemma~\ref{lem:interm2b}, $F$ is of the form
$F=\sum_{\ell=1}^L A_{\ell}\rme^{\varphi_\ell n_\ell}$,
where each $F_\ell=A_{\ell}\rme^{\varphi_\ell n_\ell}$ is an IMMF of accuracy $\varepsilon$. Proposition \ref{pro:wavcoeffIMMF} applied successively to $F_1,\cdots,F_L$ and the linearity of the continuous wavelet transform imply:
\[
\left|\partial_{b_i} c_F(a,b)-a\sum_{\ell'=1}^L \partial_i \varphi_{\ell'}(b)n_{\theta_{\ell'}(b)}\widehat{\psi}(a\nabla \varphi_{\ell'}(b))\left(A_{\ell'}(b)\rme^{\varphi_{\ell'}(b)n_{\theta_{\ell'}(b)}}\right)\right|
\leq \varepsilon a \Gamma_2(a,b)\;.
\]
Remark~\ref{rem:sep} states that under condition~(\ref{e:sep3}), the following sum
\[
\sum_{\ell'=1}^L \partial_i \varphi_{\ell'}(b)n_{\theta_{\ell'}(b)}\widehat{\psi}(a\nabla \varphi_{\ell'}(b)) \left(A_{\ell'}(b)\rme^{\varphi_{\ell'}(b)n_{\theta_{\ell'}(b)}}\right)\;,
\]
reduces to $\partial_i \varphi_{\ell}(b)n_{\theta_{\ell}(b)}\widehat{\psi}(a\nabla \varphi_{\ell}(b))\left(A_{\ell}(b)\rme^{\varphi_{\ell}(b)n_{\theta_{\ell}(b)}}\right)$.
\end{proof}
\begin{lemma}\label{lem:interm4c}
Let $F$ be a superposition of IMMF with accuracy $\varepsilon>0$ and separation $d$. Let $\ell\in\{1,\cdots,L\}$. For all $a,b$ such that
\begin{equation}\label{e:sep4}
\left|a|\nabla\varphi_{\ell}(b)|-1\right|<\Delta\;,
\end{equation}
one has, for any $i=1,2$
\begin{equation}\label{e:missing}
|\Lambda_i(a,b)-\partial_{b_i}\varphi_{\ell}(b)n_{\theta_{\ell}(b)}|\leq a\varepsilon^{1-\nu}\left(\Gamma_2(a,b)+a|\nabla\varphi_{\ell_0}(b)|\Gamma_1(a,b)\right)\;.
\end{equation}
\end{lemma}
\begin{proof}
 The proof follows the same lines than that of Theorem~\ref{th:main1}, replacing the two estimates~(\ref{e:3})  and (\ref{e:4})  on the wavelet coefficients and their partial derivatives, by~(\ref{e:lem4a}) and~(\ref{e:lem4b2}) respectively.
\end{proof}
\noindent We can now prove Theorem~\ref{th:main2}.

\noindent{\bf Proof of Theorem~\ref{th:main2}}\\
In the proof of Theorem~\ref{th:main2}, the following notations will be needed:
\[
a_{\min}=\frac{1-\Delta}{\max_{\ell\in \{1,\cdots,L\}}\sup_{b\in\mathbb{R}^2}(|\nabla\varphi_\ell(b)|)},\,a_{\max}=\frac{1+\Delta}{\min_{\ell\in \{1,\cdots,L\}}\inf_{b\in\mathbb{R}^2}(|\nabla\varphi_\ell(b)|)}\;.
\]
By assumptions on the phases $\varphi_\ell$ of the modes of $F$, $a_{\min}$ and $a_{\max}$ are both finite and positive. For any $\ell\in\{1,\cdots,L\}$, $\widetilde{\varepsilon}>0$, $b\in\R^2$ and $n\in\mathbb{S}^1$, one also defines:
\[
{\cal B}_{\ell}(\widetilde{\varepsilon},n,b)=\{k\in\R^2~;\,\max_i|k_i n-\partial_{b_i}\varphi_\ell(b) n_{\theta_{\ell}(b)}|\leq \widetilde{\varepsilon}\}\;.
\]

\noindent We now show Theorem~\ref{th:main2}. We set $\widetilde{\varepsilon}=\varepsilon^\nu$ and fix $\ell\in\{1,\cdots,L\}$. Let us first prove that
\begin{equation}\label{e:beforeinclusion}
\begin{array}{lll}
&&\lim_{\delta\to 0}\int_{\mathbb{S}^1}\int_{{\cal B}_{\ell}(\widetilde{\varepsilon},n,b)} S_{F,\varepsilon}^{\delta,\nu}(b,k,n)\rmd k\rmd n\\
&=&2\pi\int_{A_{{\varepsilon},F}(b)\cap \{a;\max_i|\Lambda_i(a,b)-\partial_{b_i}\varphi_\ell(b) n_{\theta_{\ell}(b)}|<\widetilde{\varepsilon}\}}a^{-2}c_F(a,b)\rmd a\;.
\end{array}
\end{equation}

By definition of $S_{F,\varepsilon}^{\delta,\nu}$, we have
\begin{eqnarray*}
&&\int_{\mathbb{S}^1}\int_{{\cal B}_{\ell}(\widetilde{\varepsilon},n,b)} S_{F,\varepsilon}^{\delta,\nu}(b,k,n)\rmd k\rmd n\\
&=&\int_{\mathbb{S}^1}\int_{{\cal B}_{\ell}(\widetilde{\varepsilon},n,b)} \left(\int_{A_{\varepsilon,F}(b)}c_F(a,b)\delta^{-2}\prod_{i=1}^2 h\left(\frac{k_i -{\rm Re}(\Lambda_i(a,b)~\overline{n})}{\delta}\right)\frac{\rmd a}{a^2}\right)\rmd k \rmd n\;.
\end{eqnarray*}
We first use Fubini Theorem to interchange the order of summations in this integral. Remark that:
\begin{eqnarray*}
&&\int_{\mathbb{S}^1}\int_{{\cal B}_{\ell}(\widetilde{\varepsilon},n,b)} \int_{A_{\varepsilon,F}(b)}\left|c_F(a,b)
\delta^{-2}\prod_{i=1}^2 h\left(\frac{k_i -{\rm Re}(\Lambda_i(a,b)~\overline{n})}{\delta}\right)\right|\frac{\rmd a}{a^2}\rmd k \rmd n\\
&\leq&\int_{A_{\varepsilon,F}(b)}a^{-2}|c_F(a,b)| \int_{\mathbb{S}^1}\int_{\R^2}\delta^{-2} \prod_{i=1}^2 \left| h\left(\frac{k_i -{\rm Re}(\Lambda_i(a,b)~\overline{n})}{\delta}\right)\right|\rmd k\rmd n\rmd a\\
&\leq&  2\pi\|h\|_{L^1(\R)}^2 \int_{A_{\varepsilon,F}(b)}a^{-2} |c_F(a,b)|\rmd a ~< ~+\infty\;,
\end{eqnarray*}
where the last inequality is coming from the fact that $F$ is a finite superposition of IMMFs $F_\ell$, whose wavelet coefficients satisfy~(\ref{e:remc}) and~(\ref{e:remb--}) (see Remark~\ref{rem:besov} for more details). Then by Fubini Theorem:
\begin{eqnarray*}
&&\int_{\mathbb{S}^1}\int_{{\cal B}_{\ell}(\widetilde{\varepsilon},n,b)} S_{F,\varepsilon}^{\delta,\nu}(b,k,n)\rmd k\rmd n\\
&=&\int_{A_{{\varepsilon},F}(b)}a^{-2}c_F(a,b)\left(\int_{\mathbb{S}^1}\int_{{\cal B}_{\ell}(\widetilde{\varepsilon},n,b)}\delta^{-2}\prod_{i=1}^2 h\left(\frac{k_i -{\rm Re}(\Lambda_i(a,b)~\overline{n})}{\delta}\right)\rmd k \rmd n\right)\rmd a\;.
\end{eqnarray*}
Since the integrand is bounded by:
\[
\left| a^{-2}c_F(a,b)\left(\int_{\mathbb{S}^1}\int_{{\cal B}_{\ell}(\widetilde{\varepsilon},n,b)}\delta^{-2}\prod_{i=1}^2 h\left(\frac{k_i -{\rm Re}(\Lambda_i(a,b)~\overline{n})}{\delta}\right)\rmd k \rmd n\right) \right|
\leq 2\pi\|h\|_{L^1(\R)}^2 a^{-2} |c_F(a,b)|\;,
\]
which belongs to $L^1(A_{\varepsilon,F}(b))$, the Dominant Convergence Theorem applies. Hence,
\begin{eqnarray*}
&&\lim_{\delta\to 0}\int_{\mathbb{S}^1}\int_{{\cal B}_{\ell}(\widetilde{\varepsilon},n,b)} S_{F,\varepsilon}^{\delta,\nu}(b,k,n)\rmd k\rmd n\\
&=&\int_{A_{{\varepsilon},F}(b)}a^{-2} c_F(a,b)\left(\lim_{\delta\to 0}\int_{\mathbb{S}^1}\int_{{\cal B}_{\ell}(\widetilde{\varepsilon},n,b)}\delta^{-2}\prod_{i=1}^2 h\left(\frac{k_i -{\rm Re}(\Lambda_i(a,b)~\overline{n})}{\delta}\right)\rmd k \rmd n\right)\rmd a\\
&=&\int_{A_{{\varepsilon},F}(b)}a^{-2} c_F(a,b)\left(\lim_{\delta\to 0}\int_{\mathbb{S}^1}\int_{{\cal B}_{\ell}(\widetilde{\varepsilon},n,b)\cap {\cal B}^{'}(\delta R,n,b)}\delta^{-2}\prod_{i=1}^2 h\left(\frac{k_i -{\rm Re}(\Lambda_i(a,b)~\overline{n})}{\delta}\right)\rmd k \rmd n\right)\rmd a\;
\end{eqnarray*}
where we set ${\cal B}'(\delta R,n,b)=\{k\in\R^2~;~ \max_i |k_i n-\Lambda_i(a,b)|\leq \delta R\}$, with $R$ the bound of the support of $h$.
Remark now that, if $\max_i|\Lambda_i(a,b)-\partial_{b_i}\varphi_\ell(b) n_{\theta_{\ell}(b)}|>\widetilde{\varepsilon}$, then for $\delta$ sufficiently small
\[
{\cal B}_{\ell}(\widetilde{\varepsilon},n,b)\cap {\cal B}'(\delta R,n,b)=\emptyset\;,
\]
and the limit above is equal to 0. On the other side, if $\max_i|\Lambda_i(a,b)-\partial_{b_i}\varphi_\ell(b) n_{\theta_{\ell}(b)}|<\widetilde{\varepsilon}$, for $\delta$ sufficiently small:
\[
{\cal B}_{\ell}(\widetilde{\varepsilon},n,b)\cap {\cal B}'(\delta R,n,b)={\cal B}'(\delta R,n,b)\;,
\]
and
\begin{eqnarray*}
&&\lim_{\delta\to 0}\int_{\mathbb{S}^1}\int_{{\cal B}'(\delta R,n,b)}\delta^{-2}\prod_{i=1}^2 h\left(\frac{k_i -{\rm Re}(\Lambda_i(a,b)~\overline{n})}{\delta}\right)\rmd k \rmd n\\
&&=\lim_{\delta\to 0}\int_{\mathbb{S}^1}\int_{\R^2}\delta^{-2}\prod_{i=1}^2 h\left(\frac{k_i -{\rm Re}(\Lambda_i(a,b)~\overline{n})}{\delta}\right)\rmd k \rmd n\\
&&=2\pi\;,
\end{eqnarray*}
where the last equality comes from the assumption $\int_{\mathbb{R}}h=1$. One then deduces equality (\ref{e:beforeinclusion}).

Observe now that for any $\varepsilon>0$ sufficiently small there exists some $a_0\geq a_{\max}$ such that
\begin{equation}\label{e:cond:varepsilon}
2\pi~\|F\|_{\dot{B}^{-\sigma}_{\infty,1}}a_0^{-\sigma}\leq \varepsilon^{\nu}~~~\mbox{ and }~~~\sup_{(a,b)\in (0,a_0)\times \mathbb{R}^2}\varepsilon^{1-\nu} a\left(\Gamma_2(a,b)+a|\nabla\varphi_{\ell}(b)|\Gamma_1(a,b)\right)\leq \varepsilon^{\nu}\;.
\end{equation}
Indeed, by definition of $\Gamma_1(a,b),\Gamma_2(a,b)$ given respectively in equations~(\ref{e:Gamma1}) and (\ref{e:Gamma2}), there exists $C_\Gamma >0$ not depending on $a_0$ such that
\[
\sup_{(a,b)\in (0,a_0)\times \mathbb{R}^2}a\left(\Gamma_2(a,b)+a|\nabla\varphi_{\ell}(b)|\Gamma_1(a,b)\right)\leq C_\Gamma a_0^4\;.
\]
This inequality implies that to prove~(\ref{e:cond:varepsilon}), it is sufficient to find $a_0$ such that
\begin{equation}\label{e:cond:varepsilon2}
a_{\max}\leq \left(2\pi~\|F\|_{\dot{B}^{-\sigma}_{\infty,1}}\varepsilon^{-\nu}\right)^{1/\sigma}\leq a_0\leq \left(C_\Gamma^{-1}\varepsilon^{2\nu-1}\right)^{1/4}\;.
\end{equation}
Since $\nu<1/(2+4/\sigma)$, one deduces that $\varepsilon^{(2\nu-1)/4}/\varepsilon^{-\nu/\sigma}$ tends to $\infty$ when $\varepsilon\to 0$. Hence, for $\varepsilon$ sufficiently small,
there exists some $a_0\geq a_{\max}$ satisfying~(\ref{e:cond:varepsilon2}). From now to the end of the proof, we fix such $a_0$.

We now split the integral in the right--hand side of (\ref{e:beforeinclusion}) into two terms
\begin{equation}\label{e:split}
\begin{array}{lll}
&&\lim_{\delta\to 0}\int_{\mathbb{S}^1}\int_{{\cal B}_{\ell}(\widetilde{\varepsilon},n,b)} S_{F,\varepsilon}^{\delta,\nu}(b,k,n)\rmd k\rmd n\\
&=&2\pi\int_{A_{{\varepsilon},F}(b)\cap \{a\in (0,a_0);\,\max_i|\Lambda_i(a,b)-\partial_{b_i}\varphi_\ell(b) n_{\theta_{\ell}(b)}|<\widetilde{\varepsilon}\}}c_F(a,b)\frac{\rmd a}{a^2}\\
&+&2\pi\int_{A_{{\varepsilon},F}(b)\cap \{a\geq a_0;\,\max_i|\Lambda_i(a,b)-\partial_{b_i}\varphi_\ell(b) n_{\theta_{\ell}(b)}|<\widetilde{\varepsilon}\}}c_F(a,b)\frac{\rmd a}{a^2}\;.
\end{array}
\end{equation}
We now deal with each of these two terms successively.
Recall that since by assumption $F$ is a superposition of IMMFs, then inequality~(\ref{e:remc}) holds for some $\sigma>0$ (see Remark~\ref{rem:besov}). We then get that for any $a_0$ satisfying~(\ref{e:cond:varepsilon}):
\begin{equation}\label{e:Itilde}
\begin{array}{lll}
\widetilde{I}&=&2\pi\int_{A_{{\varepsilon},F}(b)\cap \{a\geq a_0;\,\max_i|\Lambda_i(a,b)-\partial_{b_i}\varphi_\ell(b) n_{\theta_{\ell}(b)}|<\widetilde{\varepsilon}\}}c_F(a,b)\frac{\rmd a}{a^2}\\
&\leq&2\pi\int_{a_0}^{+\infty}|c_F(a,b)|\frac{\rmd a}{a^2}\leq 2\pi ~\|F\|_{\dot{B}^{-\sigma}_{\infty,1}}a_0^{-\sigma}\leq \widetilde{\varepsilon}\;.
\end{array}
\end{equation}

Let us now consider the integral
\[
2\pi\int_{A_{{\varepsilon},F}(b)\cap \{a\in (0,a_0);\,\max_i|\Lambda_i(a,b)-\partial_{b_i}\varphi_\ell(b) n_{\theta_{\ell}(b)}|<\widetilde{\varepsilon}\}}c_F(a,b)\frac{\rmd a}{a^2}\;,
\]
and prove that under assumption~(\ref{e:cond:varepsilon})
\begin{equation}
\label{egalite}
A_{{\varepsilon},F}(b)\cap \{a\in (0,a_0),\;\max_i|\Lambda_i(a,b)-\partial_{b_i}\varphi_\ell(b) n_{\theta_{\ell}(b)}|<\widetilde{\varepsilon}\}=A_{{\varepsilon},F}(b)\cap \{a\in (0,a_0);~|a|\nabla\varphi_{\ell}(b)|-1|<\Delta\}\;.
\end{equation}

Let us first assume that $a\in A_{{\varepsilon},F}(b)\cap \{a\in (0,a_0);~|a|\nabla\varphi_{\ell}(b)|-1|<\Delta\}$. Since $|a|\nabla\varphi_{\ell}(b)|-1|<\Delta$, Lemma~\ref{lem:interm4c} and assumption ~(\ref{e:cond:varepsilon}) imply that
\[
\max_i |\Lambda_i(a,b)-\partial_{b_i}\varphi_\ell(b) n_{\theta_{\ell}(b)}|\leq a\varepsilon^{1-\nu}\left(
\Gamma_2(a,b)+a|\nabla\varphi_{\ell}(b)|\Gamma_1(a,b)\right)\leq\widetilde{\varepsilon}\;,
\]
that is $a\in \{a\in (0,a_0),\;\max_i|\Lambda_i(a,b)-\partial_{b_i}\varphi_\ell(b) n_{\theta_{\ell}(b)}|<\widetilde{\varepsilon}\}$.

Conversely, let us assume that $a\in A_{{\varepsilon},F}(b)\cap \{a\in (0,a_0);\, \max_i|\Lambda_i(a,b)-\partial_{b_i}\varphi_\ell(b) n_{\theta_{\ell}(b)}|<\widetilde{\varepsilon}\}$. Since $a\in A_{{\varepsilon},F}(b)$, $|c_F(a,b)|> \widetilde{\varepsilon}$. Then by Lemma~\ref{lem:interm2b} there exists some $\ell'\in\{1,\cdots,L\}$ such that $|a|\nabla\varphi_{\ell'}(b)|-1|<\Delta$, otherwise assumption (\ref{e:cond:varepsilon}) and equation (\ref{e:lem4b}) would lead to a contradiction. In addition, remark that if $\ell'\neq \ell$, one has
\[
\max_i|\Lambda_i(a,b)-\partial_{b_i}\varphi_\ell(b) n_{\theta_{\ell}(b)}|\geq |\partial_{b_i}\varphi_\ell(b) n_{\theta_{\ell}(b)}-\partial_{b_i}\varphi_{\ell'}(b) n_{\theta_{\ell'}(b)}|-\max_i|\Lambda_i(a,b)-\partial_{b_i}\varphi_{\ell'}(b) n_{\theta_{\ell'}(b)}|
\]
As above, Lemma~\ref{lem:interm4c} and assumption ~(\ref{e:cond:varepsilon}) imply that
\[
\max_i|\Lambda_i(a,b)-\partial_{b_i}\varphi_{\ell'}(b) n_{\theta_{\ell'}(b)}|\leq \widetilde{\varepsilon}\;.
\]
Hence, we get that
\begin{eqnarray*}
\max_i|\Lambda_i(a,b)-\partial_{b_i}\varphi_\ell(b) n_{\theta_{\ell}(b)}|&\geq& d~\left[|\partial_{b_i}\varphi_{\ell}(b)|+|\partial_{b_i}\varphi_{\ell'}(b)|\right]-\widetilde{\varepsilon}\\
&\geq& \widetilde{\varepsilon}\;,
\end{eqnarray*}
provided that $\varepsilon$ is sufficiently small, namely that:
\begin{equation}\label{hyp:eps}
\varepsilon \leq \left(\frac{d}{2} \left[|\partial_{b_i}\varphi_{\ell}(b)|+|\partial_{b_i}\varphi_{\ell'}(b)|\right] \right)^{\frac{1}{\nu}}
~~~~{i.e.}~~~~d~\left[|\partial_{b_i}\varphi_{\ell}(b)|+|\partial_{b_i}\varphi_{\ell'}(b)|\right] \geq 2\tilde \varepsilon, ~~\forall \ell, \ell', \forall i
\end{equation}
Hence, necessarily, $\ell=\ell'$ and then $a\in A_{{\varepsilon},F}(b)\cap \{a\in (0,a_0),\,|a|\nabla\varphi_{\ell}(b)|-1|<\Delta\}$.

In addition, one has
\begin{eqnarray*}
&&2\pi\int_{A_{{\varepsilon},F}(b)\cap \{a\in (0,a_0),|a|\nabla\varphi_{\ell}(b)|-1|<\Delta\}}c_F(a,b)a^{-2}\rmd a\\
&=&2\pi\int_{\{a\in (0,a_0),\,|a|\nabla\varphi_{\ell}(b)|-1|<\Delta\}}c_F(a,b)a^{-2}\rmd a-2\pi\int_{\{a\in (0,a_0),|a|\nabla\varphi_{\ell}(b)|-1|<\Delta\}\setminus A_{{\varepsilon},F}(b)}c_F(a,b)a^{-2}\rmd a\;,
\end{eqnarray*}
with $|\widetilde{I}|\leq \widetilde{\varepsilon}$. Using this last equation and equations~(\ref{e:split}) and (\ref{e:Itilde}), we deduce that
\begin{eqnarray*}
&&\left|\lim_{\delta\to 0}\frac{1}{\tilde C_{\psi}} \int_{\mathbb{S}^1}\int_{{\cal B}_{\ell}(\widetilde{\varepsilon},n,b)} S_{F,\varepsilon}^{\delta,\nu}(b,k,n)\rmd k\rmd n-A_\ell(b)\rme^{\varphi_\ell(b)n_{\theta_\ell(b)}}\right|\\
&\leq&\left|\frac{2\pi}{\tilde C_{\psi}}\int_{\{a\in (0,a_0),|a|\nabla\varphi_{\ell}(b)|-1|<\Delta\}}c_F(a,b)a^{-2}\rmd a-A_\ell(b)\rme^{\varphi_\ell(b)n_{\theta_\ell(b)}}\right|\\
&&+\left|\frac{2\pi}{\tilde C_{\psi}}\int_{\{a\in (0,a_0),|a|\nabla\varphi_{\ell}(b)|-1|<\Delta\}\setminus A_{{\varepsilon},F}(b)}c_F(a,b)a^{-2}\rmd a\right|+\widetilde{\varepsilon}\;.
\end{eqnarray*}
Observe that the domain $\{a\in (0,a_0),|a|\nabla\varphi_{\ell}(b)|-1|<\Delta\}$ is bounded from below by $a_{\min}$ which only depends on $F$. In addition, if $a\not\in A_{\varepsilon,F}(b)$, inequality~(\ref{e:lem4b}) is satisfied. Hence
\begin{equation}\label{inequality}\begin{array}{lll}
\left|\int_{\{a\in(0,a_0),\,|a|\nabla\varphi_{\ell}(b)|-1|<\Delta\}\setminus A_{\widetilde{\varepsilon},F}(b)}c_F(a,b)a^{-2}\rmd a\right|&\leq& \varepsilon \left|\int_{a_{\min}}^{a_0}\left[\sup_{a\in(0,a_0)}a^2\Gamma_1(a,b)\right]a^{-2}\rmd a\right|\\
&\leq&\varepsilon\left[\sup_{a\in(0,a_0)}a^2\Gamma_1(a,b)\right]a_{\min}^{-1}\\
&\leq&\varepsilon^{1-\nu}\left[\sup_{a\in(0,a_0)}a^2\Gamma_1(a,b)\right]a_{\min}^{-1}\\
&\leq& C\widetilde{\varepsilon}\;,
\end{array}\end{equation}
where $C=\left(a_{\min}\inf_{b\in\mathbb{R}^2}|\nabla\varphi_\ell(b)|\right)^{-1}$. The last inequality comes from assumption~(\ref{e:cond:varepsilon}) valid for $\varepsilon>0$ sufficiently small. This last inequality and inequality~(\ref{e:lem4a}) proved in Lemma~\ref{lem:interm2b} then imply that
\begin{eqnarray*}
&&\left|\lim_{\delta\to 0}\frac{2\pi}{\tilde C_{\psi}}\int_{\mathbb{S}^1}\int_{{\cal B}_{\ell}(\widetilde{\varepsilon},n,b)} S_{F,\varepsilon}^{\delta,\nu}(b,k,n)\rmd k\rmd n-A_\ell(b)\rme^{\varphi_\ell(b)n_{\theta_\ell(b)}}\right|\\
&\leq&\left|\frac{2\pi}{\tilde C_{\psi}}\int_{\{a\in(0,a_0),\,|a|\nabla\varphi_{\ell}(b)|-1|<\Delta\}}\left(a\widehat{\psi}(a\nabla\varphi_\ell(b))A_\ell(b)\rme^{\varphi_\ell(b)n_{\theta_\ell(b)}}\right)a^{-2}\rmd a-A_\ell(x)\rme^{\varphi_\ell(x)n_{\theta_\ell(x)}}\right|\\
&&+\left|\frac{2\pi}{\tilde C_{\psi}}\int_{\{a\in(0,a_0),\,|a|\nabla\varphi_{\ell}(b)|-1|<\Delta\}}(\varepsilon a^2\Gamma_1(a,b))a^{-2}\rmd a\right|+(C+1)\widetilde{\varepsilon}\;.
\end{eqnarray*}
To conclude, observe now that the first term of the right--hand side vanishes. Indeed, the wavelet is real isotropic and for any $b\in\mathbb{R}^2$, $a\to \widehat{\psi}(a\nabla\varphi_\ell(b))$ is supported in  $(0,a_{\max})$. Since $a_0\geq a_{\max}$, we then have:
\[
\int_{\{a\in(0,a_0),|a|\nabla\varphi_{\ell}(b)|-1|<\Delta\}}\widehat{\psi}(a\nabla\varphi_\ell(b))\frac{\rmd a}{a}
=\int_{\R}\widehat{\psi}(a\nabla\varphi_\ell(b))\frac{\rmd a}{a}
=\frac{1}{2\pi}\int_{\mathbb{R}^2}\frac{\overline{\widehat{\psi}(\xi)}}{|\xi|^2}~\rmd\xi=\frac{\tilde C_{\psi}}{2\pi}\;,
\]
where $\tilde C_{\psi}$ has been defined in (\ref{e:synchro:id:iso}). The second term of the last inequality can be bounded using the same approach than in inequality~(\ref{inequality}). We then get Theorem~\ref{th:main2}.
\appendix
\section{Quaternionic calculus}\label{s:appendixQuater}
In this appendix, we give a short introduction to the algebra of quaternions $\mathbb{H}$. More details can be found for instance in \cite{sudbery:1979}.\\
$\mathbb{H}$ is the real 4-D vector space spanned by
$\{1, \rmi, \rmj, \rmk \}$, i.e.
a quaternion  is of the form $q=q_0+q_1\rmi+q_2\rmj+q_3\rmk$, where the algebra product is defined by $\rmi^2=\rmj^2=\rmk^2=-1$ and $\rmi\rmj=-\rmj\rmi=\rmk,\,\rmj\rmk=-\rmk\rmj=\rmi,\,\rmk\rmi=-\rmi\rmk=\rmj$. \\
\noindent For a given quaternion $q=q_0+q_1\rmi+q_2\rmj+q_3\rmk$, one defines:
\begin{itemize}
\item its real part: $\mathrm{Re}(q)=q_0$,
\item its vectorial part: $\mathrm{Vect}(q)=(q_1, q_2, q_3)\in \mathbb{R}^3$. \\
If $\mathrm{Re}(q)=0$, $q=q_1\rmi+q_2\rmj+q_3\rmk$ is called a pure quaternion.
\end{itemize}
The conjugate of the quaternion $q$ is $\overline{q}=q_0-q_1\rmi-q_2\rmj-q_3\rmk$, and its norm (or modulus) is defined by:
$$|q|=\sqrt{q\overline{q}}=\sqrt{q_0^2+q_1^2+q_2^2+q_3^2}$$
For any $(q,q')\in \mathbb{H}^2$, one has:
\[
\overline{qq'}=\overline{q'}~\overline{q}~~~and~~~~|qq'|=|q|~|q'|\;.
\]
Additionally, the quaternions are a division algebra, which means that any (non-zero) quaternion admits a multiplicative inverse given by:
\begin{equation}\label{e:inverse}
q^{-1}=\frac{\overline{q}}{|q|^2}\;.
\end{equation}
Observe that  one has:
$\overline{q^{-1}}=\frac{q}{|q|^2}=(\overline{q})^{-1}$.
The set of unit quaternion $\{q\in \mathbb{H}~;~|q|=1\}$ will be denoted by $\mathbb{S}^3$.\\
\noindent The exponential function is defined on the quaternion algebra as
\[
\exp~:q\mapsto\rme^q=\sum_{\ell=0}^{+\infty}\frac{q^\ell}{\ell!}\;,
\]
which converges since $\exp(|q|)$ converges.
Using the exponential map, each unit quaternion of $\mathbb{S}^3$ ($|q|=1$) such that $Vect(q)\neq 0$ can be written as:
\begin{equation}\label{e:polarform1}
q=(\cos\varphi+n ~\sin\varphi)=\rme^{\varphi n}~~~with~~n=\frac{Vect(q)}{|Vect(q)|},~~\cos\varphi = Re(q)~~and~~ \sin\varphi=|Vect(q)|
\end{equation}
which is an extension of the complex exponential. Notice that the usual property $e^{\mu}e^{\nu}=e^{\mu+\nu}$ is no more satisfied in general.
The polar form of any quaternion $q$ is then given by:
\begin{equation}\label{e:polarform2}
q=|q|~(\cos\varphi+n ~\sin\varphi)=|q|~\rme^{\varphi n}
\end{equation}
where $n$ is a pure unit quaternion: $n=a~\rmi+b~\rmj+ c~\rmk$ with $a^2+b^2+c^2=1$, and $\varphi\in \mathbb{R}$.
If $(\varphi,n)\in \mathbb{R}\times \mathbb{S}^3$ satisfies~(\ref{e:polarform2}), one says that $\varphi$ is a scalar argument of $q$ and $n$ a vectorial orientation of $q$.\\
\noindent A quaternion of the form $q=q_0+q_1\rmi+q_2\rmj$ with $(q_0,q_1,q_2)\in \mathbb{R}^3$ ($q_4=0$) is called a Clifford vector. Observe that if $q=a_0+a_1\rmi+a_2\rmj$ and $q'=a'_0+a'_1\rmi+a'_2\rmj$ with $(a_0,a_1,a_2), (a'_0,a'_1,a'_2)\in \mathbb{R}^3$ then
\begin{equation}\label{e:multclifford}
\overline{qq'}=\overline{q'}\times\overline{q}\;.
\end{equation}
In this case any vectorial orientation of $q$ is of the form:
\[
n=a\rmi+ b\rmj\mbox{ with }a^2+b^2=1\;.
\]
Then there exists $\theta\in\mathbb{R}$ such that $n=\cos\theta~\rmi+\sin\theta~\rmj$. This implies that $q$ admits the following polar form:
\begin{equation}\label{e:polarform3}
q=|q|~(\cos\varphi+\sin\varphi\cos\theta~\rmi+\sin\varphi\sin\theta~\rmj)=|q|~\rme^{\varphi(\cos\theta\rmi+\sin\theta\rmj)}\;.
\end{equation}
If $(\varphi,\theta)\in \mathbb{R}^2$ satisfies~(\ref{e:polarform3}), one says that $\varphi$ is a scalar argument of $q$ and $\theta$ a scalar orientation of $q$. Observe that $\varphi$ and $\theta$ are also Euler angles of the vector $\begin{pmatrix}q_0\\q_1\\q_2\end{pmatrix}\in\mathbb{R}^3$.
\section{Hilbert and Riesz Transform}\label{s:appendixRiesz}
In this appendix we briefly recall the Hilbert transform and the analytic signal analysis, and then proceed with a natural extension in two dimension,  the Riesz transform, used to define the 2-D monogenic signal in section \ref{sec:intro}.
\subsection{The 1-D Hilbert Transform}
We first define the Hilbert transform of a 1-D real valued function.
\begin{definition}
Let $f\in L^2(\mathbb{R},\mathbb{R})$. The Hilbert transform of $f$, denoted as $\mathcal{H}f$ is defined in the Fourier domain as follows: for a.e. $\xi\in\mathbb{R}$,
\[
\widehat{\mathcal{H}f}(\xi)=-\rmi~\mathrm{sgn}(\xi)\widehat{f}(\xi)\;.
\]
\end{definition}
\begin{definition}
The complex analytic signal associated to $f\in L^2(\mathbb{R},\mathbb{R})$ is then
\[
F=f+\rmi~\mathcal{H}f\;,
\]
which reads in the Fourier domain as
\[
\widehat{F}=(1+\mathrm{sgn}(\xi))\widehat{f}\;.
\]
\end{definition}
We recall below some properties of the Hilbert transform and the analytic signal.
\begin{proposition}
\;
The Hilbert transform is an antisymmetric operator on $L^2(\mathbb{R},\mathbb{R})$:
\[
<\mathcal{H}f_1,f_2>=-<f_1,\mathcal{H}f_2>\;.
\]
\end{proposition}
\begin{proposition}
\;
\begin{enumerate}
\item $f$ and $\mathcal{H}f$ are orthogonal
\item $\|F\|^2=2\|f\|^2$
\item The spectrum of the analytic signal is one sided.
\end{enumerate}
\end{proposition}
\subsection{2-D Riesz transform}
We begin with the definition of the 2-D Riesz transform of a real valued image.
\begin{definition}
Let $f\in L^2(\mathbb{R}^2)$. The Riesz transform of $f$, denoted as $\mathcal{R}f$ is the vector valued function:
\[
\mathcal{R}f=\begin{pmatrix}\mathcal{R}_1f\\\mathcal{R}_2f\end{pmatrix}\;,
\]
where for any $i=1,2$, $\mathcal{R}_if$ is defined in Fourier domain as follows: for a.e. $\xi\in\mathbb{R}^2$,
\[
\widehat{\mathcal{R}_i f}(\xi)=-\rmi~\frac{\xi_i}{|\xi|}\widehat{f}(\xi)\;.
\]
\end{definition}
We now present the key properties of $\mathcal{R}$ (\cite{stein:1970},\cite{unser:vandeville:2009}). The first ones concern the invariance with respect to dilations, translations, and the steerability property (relation with the rotations).
\begin{proposition}\label{pro:invar}
The Riesz transform commutes both with the translation, and the dilation operator: for any $f\in L^2(\mathbb{R}^2)$, $a>0$ and $b\in \mathbb{R}^2$, one has:
\[
\mathcal{R}D_a f=D_a\mathcal{R}f
~~~~
and
~~~~
\mathcal{R}T_b f= T_b\mathcal{R}f\;.
\]
\end{proposition}
\begin{proposition}\label{pro:steerability}
The Riesz transform is steerable, that is, for any $f\in L^2(\mathbb{R}^2)$ one has
\begin{equation}\label{e:steerability}
R_\theta(\mathcal{R}f)=r_\theta^{-1}\mathcal{R}(R_\theta f)=\begin{pmatrix}\cos\theta~\mathcal{R}_1(R_\theta f)+\sin\theta~\mathcal{R}_2(R_\theta f)\\-\sin\theta~\mathcal{R}_1(R_\theta f)+\cos\theta~\mathcal{R}_2(R_\theta f)\end{pmatrix}\;,
\end{equation}
where $r_\theta$ is the rotation matrix defined by~(\ref{e:rotation}).
\end{proposition}
\begin{proof}
We will prove in the Fourier domain that
$
\mathcal{R}(R_\theta f)=r_\theta R_\theta(\mathcal{R}f)\;,
$
which is equivalent to equation~(\ref{e:steerability}). Let us first remark that for a.e. $\xi\in \mathbb{R}^2$,
\begin{equation}\label{e:steer:fourier}
(\widehat{R_\theta f})(\xi)=\widehat{f}\left(r_{\theta}^{-1}\xi\right)\;.
\end{equation}
Using the definition of the Riesz transform in Fourier domain, one has, for a.e. $\xi\in\mathbb{R}^2$:
\begin{eqnarray*}
\widehat{\mathcal{R}(R_\theta f)}(\xi)&=&r_\theta r_\theta^{-1}\times \left(-\rmi \frac{\xi}{|\xi|}\widehat{f}(r_\theta^{-1}\xi)\right)\\
&=&-\rmi r_\theta\left(\frac{r_\theta^{-1}\xi}{|r_\theta^{-1}\xi|}\widehat{f}(r_\theta^{-1}\xi)\right)\\
&=&r_\theta \left(\widehat{\mathcal{R}f}\right)(r_\theta^{-1}\xi)\\
&=&r_{\theta}\left(\widehat{R_\theta\mathcal{R}f}\right)(\xi)\;,
\end{eqnarray*}
the last relation coming from~(\ref{e:steer:fourier}). One then deduces the required result.
\end{proof}
The Riesz transform is also a unitary and componentwise antisymmetric operator on $L^2(\mathbb{R}^2)$:
\begin{proposition}
\label{e:Ri}
For any $i\in\{1,2\}$, the $i$--th component of the Riesz transform $\mathcal{R}_i$ is an antisymmetric operator, namely for all $f,g\in L^2(\mathbb{R}^2)$:
\begin{equation}\label{eq:Ri}
<\mathcal{R}_i f,g>_{L^2(\mathbb{R}^2)}=-< f,\mathcal{R}_i g>_{L^2(\mathbb{R}^2)}\;.
\end{equation}
Since $\mathcal{R}_1^2+\mathcal{R}_2^2=-Id$, it implies in particular that:
\begin{equation}\label{e:R}
<\mathcal{R}f,\mathcal{R}g>_{L^2(\mathbb{R}^2,\mathbb{R}^2)}=<\mathcal{R}_1 f,\mathcal{R}_1 g>_{L^2(\mathbb{R}^2)}+<\mathcal{R}_2 f,\mathcal{R}_2 g>_{L^2(\mathbb{R}^2)}=
<f,g>_{L^2(\mathbb{R}^2)}
\end{equation}
\end{proposition}
\bibliographystyle{unsrt}
\bibliography{Synchro}
\end{document}